\documentclass[10pt]{amsart}
\usepackage{amsmath,amscd}
\usepackage{amssymb}
\usepackage{amsthm}
\usepackage{comment}
\usepackage{mathrsfs}
\usepackage[colorlinks=true]{hyperref}
\hypersetup{urlcolor=blue, citecolor=blue, linkcolor=blue}
\usepackage{calc}
\usepackage{color}

{\begin{list}{\arabic{enumi}.}{\usecounter{enumi}%
    \setlength{\labelsep}{0.5em}%
    \settowidth{\labelwidth}{(\arabic{enumi})}%
    \setlength{\leftmargin}{\labelwidth+\labelsep}}}%
{\end{list}}

\newtheorem{Theorem}{Theorem}[section]

\newtheorem{Lemma}[Theorem]{Lemma}
\newtheorem{Corollary}[Theorem]{Corollary}

\newtheorem{Proposition}[Theorem]{Proposition}

\theoremstyle{definition}

\newtheorem{Remark}[Theorem]{Remark}

\numberwithin{equation}{section}

\newcommand{\R}        {\mathbb{R}}
\newcommand{\mL}    {\mathscr{L}}
\newcommand{\D}        {\Delta}

\title[Recovery of Lower Order Non-local Perturbations]{Inverse Problems for the Fractional-Laplacian with Lower Order Non-local Perturbations}
\author[Bhattacharyya, Ghosh and Uhlmann]{S. Bhattacharyya$^{*}$, T. Ghosh$^{\dagger}$, G. Uhlmann$^{\S}$.}

\address {$*$\quad Institute for Advanced Study, 
\newline \indent\quad\ The Hong Kong University of Science and Technology,
\newline \indent\quad\ E-mail:{\tt\  arkatifr@gmail.com}}

\address{$\dagger$\quad Institute for Advanced Study, 
\newline \indent\quad\ The Hong Kong University of Science and Technology,
\newline
\newline \indent\quad\ Department of Mathematics, University of Washington
\newline
\indent\quad\ E-mail:{\tt \  imaginetuhin@gmail.com}}

\address{$\S$\quad Institute for Advanced Study, 
\newline \indent\quad\ The Hong Kong University of Science and Technology
\newline
\newline \indent\quad\ Department of Mathematics, University of Washington
\newline \indent\quad\ E-mail:{\tt \  gunther@math.washington.edu }}

\begin{document}

\begin{abstract}
In this article, we introduce a model featuring a L\'{e}vy process in a bounded domain with semi-transparent boundary, by considering the fractional Laplacian operator with lower order non-local perturbations.
We study the wellposedness of the model, certain qualitative properties and Runge type approximation.
Furthermore, we consider the inverse problem of determining the unknown coefficients in our model from the exterior measurements of the corresponding Cauchy data.
We also discuss the recovery of all the unknown coefficients from a single measurement.
\end{abstract}
\maketitle
\section{Introduction}
\subsection{Model problem and motivation:}
In this article we consider direct and inverse problems concerning a non-local  operator $\mL_{b,q}$, consisting of global and regional non-local operators.
Let $\Omega \subset \R^n$, $n\geq 1$, be a non-empty bounded, Lipschitz domain. Let us consider the operator $\mL_{b,q}$ defined by
\begin{equation}\label{operator}
\mL_{b,q} := (-\D)^t + (-\D)_{\Omega}^{{s}/{2}} \ b (-\D)_{\Omega}^{{s}/{2}} + q,\quad 0<s<t<1,
\end{equation}
where the perturbation coefficients $b$ and $q$ are $L^{\infty}(\Omega)$ functions.
For the sake of simplicity, throughout this article we consider all the functions to be real valued.
The principal part of the operator $\mL_{b,q}$ is given by the fractional Laplacian operator $(-\Delta)^{t}$ of order $2t$ whose non-locality is over the entire $\mathbb{R}^n$. The lower order term contains another non-local operator, commonly referred as the regional fractional Laplacian operator $(-\D)^{s/2}_{\Omega}$, of order $s$, whose non-locality is over $\Omega$. We also add a zero-th order local term defined in $\Omega$.
Observe that we allow any $s\in (0,t)$, so that  we can have non-local perturbations of order ($2s$) as close as we wish to the order ($2t$) of the principal part.
 
Probabilistically, the fractional Laplacian operator $(-\Delta)^t$ represents the infinitesimal generator of a symmetric $2t$-stable L\'{e}vy process in the entire space \cite{AD}.
Here we are interested in the restriction of $(-\Delta)^t$ to a bounded domain $\Omega$.
For example, one can think of the homogeneous Dirichlet exterior value problem for the fractional Laplacian operator $\left(e.g.\, (-\Delta)^tv=g\right.$ $\left.\mbox{in }\Omega\mbox{ and } v=0 \mbox{ in }\mathbb{R}^n\setminus\overline{\Omega}\right)$ which represents the  infinitesimal  generator of a symmetric $2t$-stable L\'{e}vy process for which particles are killed upon leaving the domain $\Omega$ (see \cite{BCVE, AFMR}).
The fractional Laplacian operator in $\mathbb{R}^n$ is defined by
\begin{equation}\label{FTD}
(-\Delta)^t u = \mathscr{F}^{-1} \{\lvert{ {\xi}\rvert}^{2t}  \widehat{u}(\xi) \}, \qquad u \in \mathscr{S}(\mathbb{R}^n).
\end{equation}
Here $\mathscr{S}$ denotes the Schwartz space in $\mathbb{R}^n$ and $\mathscr{F}^{-1}$ denotes the inverse Fourier transform, with the Fourier transform defined by $
\widehat{u}(\xi) = \mathscr{F} u(\xi) = \int_{\mathbb{R}^n} e^{-ix \cdot \xi} u(x) \,dx$. 
The definition \eqref{FTD} is valid for all $t>-\frac{n}{2}$.
For $0<t<1$,
the fractional Laplacian $(-\D)^t$ has an equivalent integral representation given by
\begin{equation*}
(-\D)^t u (x) = C_{n,t}\mbox{ p.v.}\int_{\R^n} \frac{u(x) - u(y)}{\lvert x-y\rvert^{n+2t}} \,dy, \quad x\in \mathbb{R}^n,
\end{equation*}
where $\mbox{p.v.}$ denotes the principal value and $C_{n,t}>0$ is a constant depending only on the dimension $n\geq 1$ and $t \in (0,1)$.

Now we define the regional fractional Laplacian, or censored fractional Laplacian operator, on the domain $\Omega$.
For $0<a<1$ we define $(-\Delta)^{a}_{\Omega}$ on $C^\infty(\overline{\Omega})$ functions as
\begin{equation*}
(-\D)^{a}_{\Omega} u (x) = C_{n,a}\mbox{ p.v.}\int_{\Omega} \frac{u(x) - u(y)}{\lvert x-y\rvert^{n+2a}} \,dy, \quad x\in \Omega.
\end{equation*}
In contrast to the fractional Laplacian, the regional fractional Laplacian $(-\Delta)^{a}_{\Omega}$
represents the infinitesimal generator of a censored $2a$-stable process that is obtained from a symmetric 
$2a$-stable L\'{e}vy process restricted to the domain $\Omega$.
The probabilistic meaning for such a process (and hence the operator) is that it can only jump within the domain. Such process can be also  defined using the Feynman-Kac formula, see \cite{CZX,MRKJ,GM1,GM2}.

We define the operator $\mL_{b,q}$ to be a weighted combination of the global and the regional fractional Laplacian operator giving a large class of processes on domains with semi-transparent boundaries.
For those processes, after hitting the boundary of $\Omega$, a particle can either go outside the domain $\Omega$ or can reflect back into the domain depending on certain parameters.
The coefficient $b(x)$ denotes the transparency or permeability. The vanishing of $b$ in $\Omega$ means that the domain is transparent/permeable, i.e. if $b \equiv 0$ in $\Omega$, then the non-local part of $\mL_{b,q} u (x)$ is $(-\D)^tu(x)$, which makes the process a $2t$-stable L\'evy process in the entire space. 
In other words, $b=0$ implies that the process can jump anywhere in the space $\R^n$ freely.

%There are lots of examples of processes taking place on a domains with semi-transparent boundary viz. diffusion through a cell membrane; Photon diffusion.
%A good example of this kind is semipermeable cell membranes \cite{Terminology,PNAS,PhysRevE,Val,Vazquez2017,Stinga}.
%A semipermeable membrane is a layer that only certain molecules can pass through. 
%Semipermeable membranes can be both biological and artificial. 
%The rate of passage depends on the pressure, concentration, and temperature of the molecules or solutes on either side, as well as the permeability of the membrane to each solute. Depending on the membrane and the solute, permeability may depend on solute size, solubility, properties, or chemistry. %How the membrane is constructed to be selective in its permeability will determine the rate and the permeability.
%Many natural and synthetic materials thicker than a membrane are also semipermeable. One example of this is the thin film on the inside of the egg.
%Artificial semipermeable membranes include a variety of material designed for the purposes of filtration, such as those used in reverse osmosis, which only allow water to pass.

Finally we mention that the study of the operator $\mL_{b,q}$ is in itself of interest mathematically, since it contains two different types of non-locality and has various interesting properties like unique continuation and Runge type approximations.
In the direct problem we prove existence and stability of solutions of $\mL_{b,q}u=F$ in $\Omega$ having the Dirichlet data $u = f$  defined on $\Omega_{e}:=\R^n \setminus \overline{\Omega}$. Here we take $F \in H^{-t}(\Omega)$ and $f \in \widetilde{H}^t(\Omega_e)$ (cf. Section \ref{fsd1}).
%For the uniqueness of solution of the above Dirichlet exterior value problem we make the following assumption.
%\begin{Assumption}\label{eva}
%We assume that $b,q \in L^{\infty}(\Omega)$ is such that
%\begin{equation}\label{eva_1}
%\mL_{b,q}\varphi=0\mbox{ in }\Omega,\quad \varphi=0 \mbox{ in }\Omega_e\quad\mbox{ has only zero solution}.
%\end{equation}
%\end{Assumption}
%Throughout this article, except Remark \ref{Rem_Cauchy}, we will assume that the operator $\mL_{b,q}$ satisfies the above assumption.
Later in this article we use the direct problem for the operator $(-\D)^a_{\Omega}$, which is defined only on $\Omega$ and also has a wellposed Dirichlet exterior value problem on a Lipschitz domain $\mathcal{O} \Subset \Omega$ in the scale of Sobolev spaces.
%We prove the wellposedness of the Dirichlet exterior value problem for the regional fractional operator $(-\D)^a_{\Omega}$ in the scale of Sobolev spaces.%
See \cite{ChenZQ,GM1,GM2} for Dirichlet exterior value problems for the regional fractional Laplacian operator on other function spaces. 
In Section \ref{Sec_direct_prob} we discuss the direct problem in detail and prove existence,
 uniqueness and stability results for the operator $\mL_{b,q}$ and the regional fractional Laplacian operator $(-\D)^a_{\Omega}$.

\subsection{Inverse problems}\label{IP}
We consider the recovery of the coefficients $b$ and $q$ in $\Omega$ from the non-local exterior data $(f, \mathscr{N}_{b,q}(f))$ measured in some open subset of $\Omega_e\times\Omega_e$. %(possibly different open subsets of $\Omega_e$ for two components of the data).
For $f \in \widetilde{H}^t(\Omega_e)$, the non-local Neumann data $\mathscr{N}_{b,q}(f)$ on $\Omega_e$ is defined by
\begin{equation*}
\mathscr{N}_{b,q}(f) := C_{n,t} \int_{\Omega} \frac{u_f(x) - u_f(y)}{|x-y|^{n+2t}}\,dy; \qquad \forall x \in \Omega_e,
\end{equation*}
where $u_f \in H^t(\R^n)$ is the unique solution of the problem $\mL_{b,q}u_f = 0$ in $\Omega$ and $u=f$ on $\Omega_e$ (see Section \ref{Sec_DN_map}).
Let $W,\widetilde{W} \subset \Omega_e$ be any two non-empty open sets. We define the non-local partial Cauchy data corresponding to the operator $\mL_{b,q}$ by
\begin{equation}\label{Cauchy}
\mathcal{C}_{b,q}(W,\widetilde{W}) := \{ (f,\mathscr{N}_{b,q}(f)|_{\widetilde{W}});\, f \in \widetilde{H}^t(W) \}.
\end{equation}

We assume that $b\in L^{\infty}(\Omega)$, $q \in L^{\infty}(\Omega)$ are compactly supported functions and we also assume throughout this paper that they satisfy the following eigenvalue condition
\begin{equation}\label{eva}
\mL_{b,q}\varphi=0\mbox{ in }\Omega,\quad \varphi=0 \mbox{ in }\Omega_e\quad\mbox{ has only the zero solution}.
\end{equation}
For non-negative functions $b$ and $q$ the above eigenvalue condition is satisfied. 
Let $W,\widetilde{W}\subset \Omega_e$ be any two non-empty open subsets of $\Omega_e$. We prove two results (Theorems \ref{Main_Th_1} and \ref{Main_Th}) regarding the recovery of the coefficients $b,q$ in $\Omega$ from the non-local partial Cauchy data $\mathcal{C}_{b,q}(W,\widetilde{W})$.
In the first result (Theorem \ref{Main_Th_1}) we show that measuring $(f, \mathscr{N}_{b,q}(f)|_{\widetilde{W}})$ for all $f \in \widetilde{H}^t(W)$, one can uniquely determine $b$ and $q$ in $\Omega$.
We discuss also how much one can recover from a single measurement of the non-local Cauchy data. We prove that given a single measurement of the Cauchy data $\mathcal{C}_{b,q}(W,\widetilde{W})$ corresponding to a single non-zero $f$, one can determine $b$ and $q$ on certain subsets of $\Omega$.
Moreover, given a single measurement $(f, \mathscr{N}f|_{\widetilde{W}})$, we show that the above subsets of $\Omega$ (where we can recover $b$ and $q$) are optimal in the sense that one cannot conclude anything about $b$ and $q$ outside those subsets from that measurement.
%In the process, we prove an unique continuation principal and Runge type approximation properties for the operator $\mL_{b,q}$ as well as for $(-\D)^a_{\Omega}$ (see Theorem \ref{thm2}).
We state our results for the inverse problems in Theorems \ref{Main_Th_1} and \ref{Main_Th} and prove them in Section \ref{Sec_inv_prob}.

These type of inverse problems are often addressed as Calder\'{o}n type inverse problems. In the standard Calder\'{o}n problem \cite{Calderon1980} the objective is to determine the electrical conductivity of a medium from voltage and current measurements made on its boundary.
Study of the inverse boundary value problems have a long history, in particular, in the context of electrical impedance tomography; it has applications on seismic and medical imaging, as well as to inverse scattering problems.
We refer to \cite{UG} and the references therein for a survey of this topic.

The study of Calder\'{o}n type inverse problems for non-local operators began with the article \cite{GSU}, where the authors address the inverse problem of determining the potential $q$ in the fractional Schr\"{o}dinger operator $\left((-\Delta)^t+q(x)\right)$, $0<t<1$, in $\Omega$ from the corresponding Dirichlet to Neumann map in the exterior domain $\Omega_e$. 
In \cite{RS1} the authors study stability estimates for recovering the potential $q$. Later it has been shown that with a single measurement of $(f,\mathscr{N}_{b,q}(f))$ for non zero $f$, it is possible to recover and reconstruct the potential $q$ in $\Omega$ (see \cite{GRSU}). 
The problem of recovering $q$ for the anisotropic fractional elliptic operators has been considered in \cite{GLX}.

\subsubsection{All Measurements}
In this article we are interested in determining two unknowns $b$ and $q$ in the perturbed non-local operator $\mL_{b,q}$ where in addition  to  the zeroth order perturbation we have a $2s$-order non-local perturbation to the  fractional Laplacian operator of order $2t$,  $0<s<t<1$.
Our result for all measurements is:
\begin{Theorem}\label{Main_Th_1}
Let $\Omega \subset \R^n$, $n\geq 1$, be a bounded Lipschitz domain. Let $b_1,b_2,q_1,q_2\in L^{\infty}(\Omega)$ compactly supported in $\Omega$ be such that Assumption \eqref{eva} is satisfied for $\mL_{b_k,q_k}$, $k=1,2$.
Let $W, \widetilde{W} \subset \Omega_e$ be any two non-empty open subsets.\\
\noindent If $\mathcal{C}_{b_1,q_1}(W,\widetilde{W}) = \mathcal{C}_{b_2,q_2}(W,\widetilde{W})$, then $q_1=q_2$ and $b_1=b_2$ in $\Omega$.
\end{Theorem}

\subsubsection{Single measurement}
We also consider the problem of recovering $b$ and $q$ on suitable subsets of $\Omega$ subject from the non-local Cauchy data $\mathcal{C}_{b_2,q_2}(W,\widetilde{W})$ for only a single non-zero $f \in \widetilde{H}^t(\Omega_e)$.
Let us fix a non-zero $f \in \widetilde{H}^t(\Omega_e)$ and let $u_f$ be the unique solution of the problem $\mL_{b,q}u_f=0$ in $\Omega$ and $u_f|_{\Omega_e}=f$.
If $b \equiv 0$ in $\Omega$, a single non-zero measurement of the Cauchy data is enough to determine $q$ in $\Omega$ (see \cite{GRSU}).
If $b\not\equiv 0$ in $\Omega$ and $u_f = 0$ in some non-empty open subset $E \subset \Omega$, then we cannot conclude anything about $q$ in $E$.
More precisely, if $b \not\equiv 0$ in $\Omega$ and there exists a non-empty open set $E\subset \Omega$ such that $u_f=0$ in $E$, then for any $\varphi \in C_c(E)$
\begin{equation}\label{invariance}
\mL_{b,q}u_f = 0 \quad \mbox{in } \Omega \ \implies \mL_{b,(q+\varphi)}u_f = 0 \quad \mbox{in } \Omega,
\end{equation}
with $\mathcal{C}_{b,q}(W,\widetilde{W}) = \mathcal{C}_{b,q+\varphi}(W,\widetilde{W})$. 
Therefore, it is impossible to recover $q$ on $E$ from the single measurement $(f,\mathscr{N}_{b,q}(f)|_{\widetilde{W}})$.
Similarly, if $(-\D)^{s/2}_{\Omega}u_f = 0$ in some non-empty open subset $F \subset \Omega$, then it is impossible to recover $b$ on $F$ from the single measurement  $(f,\mathscr{N}_{b,q}(f)|_{\widetilde{W}})$.
Therefore, the optimal result would be  to recover $b$ and $q$ on the support of $(-\D)^{s/2}_{\Omega}u_f$ and $u_f$ respectively. 
%In Theorem \ref{Main_Th} we formally state this result.
%
%It is interesting to observe that even one measurement of the Cauchy data can determine two unknown parameters simultaneously.
%If for some $f\in \widetilde{H}^t(\Omega)$, the supports of $b$ and $q$ are contained in the support of $(-\D)^{s/2}_{\Omega}u_f$ and $u_f$ in $\Omega$ respectively, then the complete recovery of $b$ and $q$ follows from that single measurement. 
%If not, all the measurements (i.e. varying $f$ in $\widetilde{H}^t(W)$) provides the recovery of the unknown coefficients $b,q$. In addition, two suitable measurements  $(f_l, \mathscr{N}_{b,q}u_{f_l})$, $l=1,2$ can determine the coefficients completely in $\Omega$ provided $u_{f_l}$ and $(-\D)^{s/2}_{\Omega}u_{f_l}$ have disjoint zero sets corresponding to $f_1$ and $f_2$.
Our result for a single measurement is:

\begin{Theorem}\label{Main_Th}
Let $\Omega \subset \R^n$, $n\geq 1$, be a bounded Lipschitz domain and $b_1,b_2,q_1,q_2\in C_c(\Omega)$ continuous functions with compact support inside $\Omega$.
We assume $\mL_{b_1,q_1}$, $\mL_{b_2,q_2}$ are such that the assumption \eqref{eva} is satisfied.
Let $f \in \widetilde{H}^t(\Omega_e)$ be a fixed non-zero function and $(u_f)_j\in H^t(\mathbb{R}^n)$ solves $\mathscr{L}_{b_j,q_j}(u_f)_j = 0$ in $\Omega$ with $(u_f)_j = f$ in $\Omega_e$ for $j=1,2$.\\
\noindent If $\mathscr{N}_{b,q}(f)_1 = \mathscr{N}_{b,q}(f)_2$ on $\widetilde{W}$, where $\widetilde{W} \subset \Omega_e$ is some nonempty open subset, then $(u_f)_1 \equiv (u_f)_2$ in $\R^n$.
Moreover, we have $q_1=q_2$ on the support of $(u_f)_1$ and $b_1=b_2$ on the support of $(-\D)^{s/2}_{\Omega} (u_f)_1$ in $\Omega$.
\end{Theorem}

\noindent The remainder of this paper is organized as follows:
In Section \ref{Sec_2} we discuss the direct problem for the regional fractional Laplacian operator $(-\D)^a_{\Omega}$ and the non-local operator $\mL_{b,q}$. We first define the fractional order Sobolev spaces on $\R^n$, as well as on a non-empty open domain. 
%Following that we extend the global and the regional fractional Laplacian operators to the respective fractional order Sobolev spaces.%
Then we study wellposedness of the Dirichlet exterior value problem for the regional fractional Laplacian operator acting on fractional order Sobolev spaces.
%Here we mention that we also discuss wellposedness of a Dirichlet boundary value problem for the regional fractional Laplacian operator in Appendix \ref{d1}.
In the rest of the section we discuss wellposedness of the Dirichlet exterior value problem for the non-local operator $\mL_{b,q}$ and finally, we finish our discussion on direct problems by defining the non-local Cauchy data corresponding to the operator $\mL_{b,q}$.
In Section \ref{sec3} we prove theorems \ref{Main_Th_1} and \ref{Main_Th}.
Having the equality of the non-local Cauchy data corresponding to two non-local operators $\mL_{b_k,q_k}$, for $k=1,2$, and the unique continuation principle for the fractional Laplacian operator $(-\D)^t$ (c.f. Proposition \ref{sucp}), we derive an integral identity \eqref{int_id_2} relating the perturbation coefficients $b_k$ and $q_k$, for $k=1,2$.
Next we discuss the unique recovery of the coefficients from the integral identity \eqref{int_id_2}.
In order to do that we prove unique continuation (Lemma \ref{sucpl}) and a Runge type approximation property (Lemma \ref{runge_0}) for the regional fractional Laplacian operator. Using these we complete the proof of Theorem \ref{Main_Th}.
Finally we prove a Runge type approximation (Lemma \ref{sqml}) for the non-local operator $\mL_{b,q}$ and using that we complete the proof of Theorem \ref{Main_Th_1}.

\section{Direct Problems}\label{Sec_2}
In this section we study the direct problems for the operators $(-\D)^{a}_{\Omega}$ and $\mL_{b,q}$.
For simplicity we assume that all functions including $b$ and $q$ to be real valued.
Let us start with recalling the definition of the fractional order Sobolev spaces $H^r(\R^n)$, for $r \in \R$.
We define the space $H^r(\R^n)$ as
\begin{equation}\label{def1}
H^r(\R^n) := \{ u \in \mathscr{S}^\prime(\R^n) ; \langle \xi \rangle^{r} \widehat{u}(\xi) \in L^2(\R^n) \}, \quad \forall r\in \mathbb{R};
\end{equation}
equipped with the norm $\lVert u \rVert_{H^r(\R^n)} := \lVert \langle \xi \rangle^r \hat{u}(\xi) \rVert_{L^2(\R^n)}$,
where $\mathscr{S}^\prime(\R^n)$ denotes the space of tempered distributions, $\langle \xi \rangle = (1 + \lvert \xi \rvert^2)^{\frac{1}{2}}$ and $\widehat{\cdot}$ denotes the Fourier transform.
The space $H^{-r}(\R^n)$, for $r>0$, can also be realized as the dual of the fractional order Sobolev space $H^r(\R^n)$, i.e. $H^{-r}(\R^n) = \left(H^r(\R^n)\right)^*$.
 
Here we also provide an equivalent definition of the fractional order Sobolev spaces that does not use the Fourier transform as in \eqref{FTD}.
For $r\in (0,1)$ one can equivalently define the space $H^r(\R^n)$ as
\begin{equation}\label{key}
H^{r}(\R^n) = \{ u \in L^2(\R^n) ; \frac{u(x) - u(y)}{\lvert x-y\rvert^{\frac{n}{2}+r}} \in L^2(\R^n \times \R^n) \},
\end{equation}
with the well-known Aronszajn-Slobodeckij inner product, given for real-valued $u,v$ by (\cite{AF}):
\begin{equation*}
\langle u,  v\rangle_{H^r(\R^n)}
= \int_{\R^n} u(x)v(x)\,dx + \int_{\R^n} {\int_{\R^n} {\frac{(u(x) - u(y))(v(x) - v(y))}{\left|x-y\right|^{n+2r}} \, dy} \, dx},
\end{equation*}
for all $u,v\in H^{r}(\mathbb{R}^n)$.
Following that, we can assign the graph norm on $H^r(\mathbb{R}^n)$ by
\begin{equation*}
\lVert u \rVert^2_{H^r(\R^n)} = \lVert u \rVert^2_{L^2(\R^n)} + \left\lVert \frac{u(x) - u(y)}{\left|x-y\right|^{\frac{n}{2}+r}} \right\rVert^2_{L^2(\R^n\times\R^n)}.
\end{equation*}

%By that we mean for $F \in H^{-r}(\R^n)$ there exist a $f,g \in L^2(\R^n)$ such that
%\begin{equation*}
%\begin{aligned}
%&\langle F, \varphi \rangle_{(H^{-r}(\R^n),H^{r}(\R^n))} 
%= \langle g, \varphi \rangle_{L^2(\R^n)} + \langle f, (-\Delta)^{r/2}\varphi \rangle_{L^2(\R^n)}\\
%&= \int_{\R^n}g(x)\varphi(x)\, dx + \int_{\R^n} {\int_{\R^n} {\frac{(f(x) - f(y))(\varphi(x) - \varphi(y))}{\left|x-y\right|^{\frac{n}{2}+r}} \, dy} \, dx},\quad \forall\varphi \in \mathscr{S}(\R^n).
%\end{aligned}
%\end{equation*}

\subsection{Fractional Laplacian operator}
Let $0<t<1$, we have defined the fractional Laplacian $(-\Delta)^t$ in $\mathbb{R}^n$ for Schwartz class functions in \eqref{FTD}.
%\begin{equation*}
%\forall x\in\mathbb{R}^n,\quad(-\Delta)^a u(x) = \mathscr{F}^{-1} \{ \lvert{\xi}\rvert^{2a} \widehat{u}(\xi) \}, \quad u \in %\mathscr{S}(\mathbb{R}^n).
%\end{equation*}
Note that, $(-\Delta)^{t}u$ is not a Schwartz class function due to its lack of decay near infinity, in particular, $(-\Delta)^{t}u$ decays at infinity as $|x|^{-n-2t}$ (see \cite{LNS}).
The operator $(-\D)^t$ satisfies the following integration by parts formula on $\mathbb{R}^n$ 
in the $L^2$ sense (i.e. $(-\Delta)^t u\in L^2(\mathbb{R}^n), \, u\in\mathscr{S}(\mathbb{R}^n)$ for $0<t<1$) 
given by 
\begin{equation}\label{int_by_parts}
\int_{\mathbb{R}^n} \left((-\Delta)^t u\right)\, v\, dx
=  \int_{\mathbb{R}^n} \left((-\Delta)^{t/2} u\right)\, \left((-\Delta)^{t/2}v\right)\, dx,\quad \forall  u,v\in \mathscr{S}(\mathbb{R}^n).
\end{equation}
Therefore,
\begin{equation}
\int_{\mathbb{R}^n} \left((-\Delta)^t u\right)\, v\, dx
=  \int_{\mathbb{R}^n} \left((-\Delta)^{t} v\right)\, u\, dx,\quad \forall  u,v\in \mathscr{S}(\mathbb{R}^n).
\end{equation}
There are several equivalent definitions of the fractional Laplacian, see \cite{KWM}. For instance, it can be defined using the principal value integral as  ($0 < t< 1$)
\begin{equation*}
\begin{aligned}
(-\Delta)^t u(x) &= C_{n,t} \ \mathrm{p.v.} \int_{\mathbb{R}^n} \frac{u(x)-u(y)}{\lvert{x-y}\rvert^{n+2t}} \,dy\\
&= C_{n,t}\lim\limits_{\epsilon \to 0^+} \int_{\mathbb{R}^n\setminus B(x,\epsilon)} {\frac{u(x)-u(y)}{\left|x-y\right|^{n+2t}} dy}, 
\end{aligned}
\end{equation*}
where $C_{n,t}$ is a constant given by $\frac{4^{t}\Gamma(\frac{n}{2}+t)}{\pi^{n/2}\Gamma(-t)}$ (see \cite{Hitchhiker}), and $B(x,\epsilon)$ is a ball in $\mathbb{R}^n$ centered at $x$ with radius $\epsilon>0$.
The difference $u(x) -u(y)$ in the numerator, which vanishes at the singularity of the kernel, provides a regularization. This together with the averaging of positive and negative parts allows the principal value to exist at least for smooth $u$ with sufficient decay.
However, when $t \in (0,\frac{1}{2})$, the integral is not singular near $x$.
Indeed, for $u \in \mathscr{S}(\mathbb{R}^n)$ and $0<t<\frac{1}{2}$, we have
\begin{align*}
&\left| \int_{\mathbb{R}^n} {\frac{u(x)-u(y)}{\left|x-y\right|^{n+2t}} dy} \right|\\
& \leq 2\left(\| \nabla u \|_{L^\infty} \int_{\overline{B}(x, 1)} {\frac{dy}{\left|x-y\right|^{n+2t-1}}} + \| u \|_{L^\infty} \int_{\mathbb{R}^n \setminus \overline{B}(x, 1)} {\frac{dy}{\left|x-y\right|^{n+2t}}}\right),
\end{align*}
and both of the integrals in the right hand side are finite. Note that here we have used only $C^1$ regularity and the boundedness of the gradient of $u$. 
Moreover, by using the $C^2$ regularity and boundedness of the second order derivatives as well, in general for $t \in (0,1)$ we can  write the fractional Laplacian as a standard Lebesgue integral given by (see \cite{BCVE})
\begin{align*}
%&\forall x\in\mathbb{R}^n,\\
&\left(- \Delta \right)^{t}u(x) = -\frac{C_{n,t}}{2} \int_{\mathbb{R}^n} {\frac{u(x+y)+u(x-y) - 2u(x)}{\left|y\right|^{n+2t}} \, dy},\quad u\in\mathscr{S}(\mathbb{R}^n), x\in\mathbb{R}^n.
\end{align*}
Next we extend $(-\Delta)^t$ to larger spaces, in particular to Sobolev spaces.
\begin{Proposition}\label{Cont}
The fractional Laplacian extends as a bounded map 
\begin{equation}\label{mapFL}
(-\Delta)^t: \ H^r(\mathbb{R}^n) \rightarrow H^{r-2t}(\mathbb{R}^n)
\end{equation}
whenever $r \in \mathbb{R}$ and $t\in (0,1)$.
\end{Proposition}
%\noindent The proof of Proposition \ref{Cont} directly follows from the fact that the fractional Laplacian operator $(-\D)^a$ is a pseudodifferential operator of order $2a$, which is easy to see from the definition \eqref{FTD}.

\subsection{Fractional Sobolev spaces on domains}\label{fsd1}
Now we briefly discuss about the fractional order Sobolev spaces defined on any open subset with Lipschitz boundary.
Let $\mathcal{O} \subseteq \R^n$ be any non-empty open set (bounded or unbounded) with Lipschitz boundary. 
%Let open subset $\mathcal{O}$ of $\R^n$ definition of regional fractional Laplacian in Sobolev spaces.
We define $H^r(\mathcal{O})$, for $r \in \R$ by
\begin{equation*}
H^{r}(\mathcal{O}) := \{ u|_{\mathcal{O}} ; u \in H^r(\R^n) \},
\end{equation*}
equipped with the norm
\begin{equation*}
\lVert u \rVert_{H^r(\mathcal{O})}:= \inf_{ \substack{v \in H^r(\R^n),\\ v|_{\mathcal{O}} = u}} \lVert v \rVert_{H^r(\R^n)}.
\end{equation*}
One can equivalently define the space $H^r(\mathcal{O})$ for $r \in (0,1)$ by
\begin{equation*}
H^{r}(\mathcal{O}) = \{ u \in L^2(\mathcal{O}) ; \frac{u(x) - u(y)}{\lvert x-y\rvert^{\frac{n}{2}+r}} \in L^2(\mathcal{O} \times \mathcal{O}) \},
\end{equation*}
equipped with the inner product $\langle\, ,\rangle_{r}$ for real-valued $u,v$ as
\begin{equation*}
\langle u, v\rangle_r = \int_{\mathcal{O}}uv\, dx + \int_{\mathcal{O}\times\mathcal{O}} \frac{\left(u(x) - u(y)\right)\left(v(x)-v(y)\right)}{\lvert x-y\rvert^{n+2r}}\, dx\,dy.
\end{equation*}
Let $C^\infty_c(\overline{\mathcal{O}})$ denotes the restriction of all $C^\infty_c(\mathbb{R}^n)$ (compactly supported smooth functions in $\mathbb{R}^n$) functions to $\overline{\mathcal{O}}$.
We note that $C^{\infty}_c(\overline{\mathcal{O}})$ is dense in $H^r(\mathcal{O})$. %with respect to the above norm ($\lvert| u|\rvert_{H^r(\mathcal{O})}=\langle u, u\rangle_r$).

\noindent The dual of $H^r(\mathcal{O})$ is
\begin{equation*}
\left(H^r(\mathcal{O})\right)^{*} =\{u\in H^{-r}(\mathbb{R}^n)\, ;\, supp\, u\subseteq\overline{\mathcal{O}}\}.
\end{equation*}
%Let $H^r_0(\mathcal{O})$ denotes the closure of $C^\infty_c(\mathcal{O})$ %functions in $H^r(\mathcal{O})$. Then the dual of $H^r_0(\mathcal{O})$ is %$H^{-r}(\mathcal{O})=\{u|_{\mathcal{O}}:\, u\in H^{-r}(\mathbb{R}^n)\}$. %Moreover, whenever $0<r<\frac{1}{2}$, we have   $H^r_0(\mathcal{O})= H^r(\mathcal{O})$.  
%We also denote $\widetilde{H}^r(\mathcal{O})$ as the closure of %$C^\infty_c(\mathcal{O})$ functions in $H^r(\mathbb{R}^n)$ and for %\mathcal{O}$ being Lipschitz we identity the space as %$\widetilde{H}^r(\mathcal{O})= H^r_0(\mathcal{O})$ for $0<r<1$.
We define
\[\begin{cases}
H^r_0(\mathcal{O}):=\{\mbox{closure of $C^\infty_c(\mathcal{O})$ in $H^r(\mathcal{O})$}\}\\[1mm]
\widetilde{H}^r(\mathcal{O}) := \{\mbox{closure of $C^\infty_c(\mathcal{O})$ in $H^r(\mathbb{R}^n)$}\}.
\end{cases}\]
Then we have the following identifications \cite{McLean}: 
\begin{equation}\left(\widetilde{H}^r(\mathcal{O})\right)^* = H^{-r}(\mathcal{O}) \text{ and } \left(H^r(\mathcal{O})\right)^* = \widetilde{H}^{-r}(\mathcal{O}), \quad r \in \mathbb{R}. \label{id}\end{equation}
Furthermore, one has (\cite{McLean})
\begin{eqnarray}
%\widetilde{H}^r(\mathcal{O})= \{u\in H^r(\mathbb{R}^n);\, supp\, u\subseteq \overline{\mathcal{O}}\}, \quad r \in \mathbb{R},\\[1mm]
\widetilde{H}^r(\mathcal{O})=H^r(\mathcal{O}) =  H^r_0(\mathcal{O}),\,\,  r < 1/2,\quad\mbox{ and }\quad H^{1/2}(\mathcal{O}) =  H^{1/2}_0(\mathcal{O}),\label{key_1}\\ 
\widetilde{H}^r(\mathcal{O})=H^r_0(\mathcal{O}), \,\, r>-\frac{1}{2},\, r\neq \{\frac{1}{2},\frac{3}{2},\cdots\}.\label{key2}
\end{eqnarray}
We remark here that the above equivalences require $\mathcal{O}\subset\mathbb{R}^n$ to be  a Lipschitz domain in $\R^n$.
\begin{Remark}\label{prerem}
	Here we note that, the characteristic function on the set $\mathcal{O}$ namely $\chi_{\mathcal{O}}\in H^{1/2}_0(\mathcal{O})$ but $\chi_{\mathcal{O}}\neq \widetilde{H}^{1/2}(\mathcal{O})$.
\end{Remark}
\noindent
Next we define the Lions-Magenes space $H^{1/2}_{0,0}(\mathcal{O})$ (see \cite[Chapter 33]{Tartar}) as
\[H^{1/2}_{0,0}(\mathcal{O}):=\{u\in H^{1/2}(\mathcal{O})\,;\, \frac{u(x)}{d(x,\mathcal{O}^c)^{1/2}}\in L^2(\mathcal{O})\}\]
where $d(x) = d(x, \mathcal{O}^c)$ is a smooth positive extension of the distance to boundary function $dist(x,\mathcal{O}^c)$ inside $\mathcal{O}$.
We have  the following equivalence 
\begin{equation*}
\widetilde{H}^{1/2}(\mathcal{O})= H^{1/2}_{0,0}(\mathcal{O}).
\end{equation*}
We also mention that (see \cite[Lemma 37.1]{Tartar})
\begin{equation}\label{tareq}
u\in H^r(\mathcal{O}) \mbox{ and }\frac{u}{(d(x,\mathcal{O}^c))^r}\in L^2(\mathcal{O}) \Longleftrightarrow u\in \widetilde{H}^r(\mathcal{O}),\quad r\in (0,1).
\end{equation}
%where $d(x,\mathcal{O}^c)$ is same as above. 
\subsection{Regional fractional Laplacian operator}
Let us consider a Lipschitz domain $\mathcal{O}\subset\mathbb{R}^n$. Let $0<a<1$, 
recall the definition of the regional fractional Laplacian operator $(-\Delta)^a_{\mathcal{O}}$ on the domain $\mathcal{O}$ over the class of $C^\infty_c(\overline{\mathcal{O}})$ functions by 
\begin{equation}\label{RFL}
(-\D)^a_{\mathcal{O}}u(x) = C_{n,a}\lim\limits_{\epsilon \to 0^+} \int_{\mathcal{O} \setminus B(x,\epsilon)} \frac{u(x) - u(y)}{\lvert x-y\rvert^{n+2a}} \,dy, \quad u\in C^\infty_c(\overline{\mathcal{O}}), x\in\mathcal{O}.
\end{equation}
Now we state the following proposition (\cite[Theorem 3.3]{GM2}) for the regional fractional Laplacian operator.
\begin{Proposition}[\cite{GM2}, Theorem 3.3]
For all $u,v\in C^\infty_c(\overline{\mathcal{O}})$ we have
\begin{equation}\label{R_int_by_parts} \int_{\mathcal{O}} \left((-\Delta)^a_{\mathcal{O}} u\right)\, v\, dx = \frac{C_{n,a}}{2} \int_{\mathcal{O}}\int_{\mathcal{O}} {\frac{(u(x) - u(y))(v(x) - v(y))}{\left|x-y\right|^{n+2a}} \, dy} \, dx.
\end{equation}
\end{Proposition}

From the integration by parts formula \eqref{R_int_by_parts} we get the following corollary.
\begin{Corollary}
\[\int_{\mathcal{O}} \left((-\Delta)^a_{\mathcal{O}} u\right)\, v\, dx 
=  \int_{\mathcal{O}} \left((-\Delta)^a_{\mathcal{O}} v\right)\, u\, dx, \quad \forall u,v\in C^\infty_c(\overline{\mathcal{O}}).\]
%In particular, when $\mathcal{O}$ is also bounded, then the above integration by parts identity holds for all $u,v\in C^\infty(\overline{\mathcal{O}})$. 
\end{Corollary}

Clearly, when $\mathcal{O}=\mathbb{R}^n$  the regional fractional Laplacian 
coincides with the definition of the usual fractional Laplacian $(-\Delta)^a$,\, $(0<a<1$). Moreover, for $u\in C^\infty_c(\mathcal{O})$ the regional fractional Laplacian can be identified with the fractional Schr\"{o}dinger operator $\left((-\Delta)^a - \varphi_{a}\right)$ in $\mathcal{O}$ $(0<a<1)$
\begin{equation}\label{local_nonlocal}\forall x\in \mathcal{O},\quad (-\Delta)^a_{\mathcal{O}}u(x) = (-\Delta)^a u(x) - \varphi_{a}(x) u(x), \quad \forall u\in C^\infty_c(\mathcal{O})\end{equation}
where
\[ \varphi_{a}(x) = C_{n,a} \int_{\mathbb{R}^n\setminus\mathcal{O}}\frac{1}{|x-y|^{n+2a}}\, dy. 
 \]
The potential $\varphi_a\in C^{0,1}_{loc}(\mathcal{O})$  is a locally Lipschitz function and for some constant $C>1$ (see \cite[Lemma 2.4]{CH})
\begin{equation}\label{est_phi}
\frac{1}{C}\left(\text{dist }\ (x, \mathcal{O}^c)\right)^{-2a} \leq \varphi_a(x) \leq 
C\left(\text{dist }\ (x,\mathcal{O}^c)\right)^{-2a},\quad x\in\mathcal{O}. 
\end{equation} 

We extend the definition of $(-\Delta)^{a}_{\mathcal{O}}$ over the space $H^{a}(\mathcal{O})$ for $0<a<1$. 
Since $C^\infty_c(\overline{\mathcal{O}})$ is dense in $H^a(\mathcal{O})$, for $u,v\in H^a(\mathcal{O})$ we define $(-\Delta)^{a}_{\mathcal{O}}u \in H^{-a}(\mathcal{O})$ weakly by
\begin{equation}\label{rev2} 
\langle (-\D)^a_{\mathcal{O}} u, v\rangle_{(H^{-a}(\mathcal{O}),H^{a}(\mathcal{O}))} 
= \frac{C_{n,a}}{2} \int_{\mathcal{O}}\int_{\mathcal{O}} {\frac{(u(x) - u(y))(v(x) - v(y))}{\left|x-y\right|^{n+2a}} \, dy} \, dx.
\end{equation}

\begin{Proposition}\label{Prop_cont}
Let $0<a<1$ and $\mathcal{O}\subset \R^n$ be an open subset with Lipschitz boundary, then
\begin{equation}\label{acon}
(-\Delta)^a_{\mathcal{O}}: H^a(\mathcal{O})\rightarrow H^{-a}(\mathcal{O}) \mbox{ is continuous.}
\end{equation}
\end{Proposition}
\begin{proof}
From the integration by parts formula \eqref{R_int_by_parts} we get the duality inner-product as \eqref{rev2}, satisfying
\begin{equation*}
\begin{aligned} \left|\langle (-\D)^a_{\mathcal{O}} u, v\rangle_{(H^{-a}(\mathcal{O}),H^{a}(\mathcal{O}))} \right|
&\leq \left\|\frac{\left(u(x)-u(y)\right)}{|x-y|^{\frac{n}{2}+a}}\right\|^2_{L^2(\mathcal{O}\times\mathcal{O})} \left\|\frac{\left(v(x)-v(y)\right)}{|x-y|^{\frac{n}{2}+a}}\right\|^2_{L^2(\mathcal{O}\times\mathcal{O})}\\
&\leq \|u\|_{H^a(\mathcal{O})}\|v\|_{H^a(\mathcal{O})}.
\end{aligned}
\end{equation*}
\end{proof}

Moreover, from \eqref{rev2} it is also clear that for $0\leq \delta < a$,
\begin{equation}\label{acon2} (-\Delta)^a_{\mathcal{O}}: H^{a+\delta}(\mathcal{O})\rightarrow H^{-a+\delta}(\mathcal{O}) \mbox{ is continuous.}\end{equation}
Next we examine for $u\in \widetilde{H}^a(\mathcal{O})$ $(0<a<1)$, whether $(-\Delta)^{a/2}_{\mathcal{O}}u\in L^2(\mathcal{O})$. 
%First we see
\begin{Lemma}\label{lemnew}
For all ${\mathcal{O}^{\prime}} \Subset \mathcal{O}$ and for all $u \in \widetilde{H}^a(\mathcal{O})$, $0<a<1$ we have
\begin{equation}\label{nza} (-\D)^{a/2}_{\mathcal{O}}u \in L^2(\mathcal{O}^{\prime}).
\end{equation}
\end{Lemma}

\begin{proof}
Let us extend the function $u\in \widetilde{H}^a(\mathcal{O})$ by $0$ in $\mathbb{R}^n$ and denote the extension by $u$ also, $u\in H^a(\mathbb{R}^n)$. 
Therefore, we have $(-\Delta)^{a/2}u\in L^2(\mathbb{R}^n)$ (see \eqref{mapFL}), in particular  $(-\Delta)^{a/2}u\in L^2(\mathcal{O})$. Now, from \eqref{local_nonlocal}, in $\mathcal{O}$ we have 
\begin{equation}\label{ln2}
(-\Delta)^{a/2}_{\mathcal{O}}u +\varphi_{a/2}(x)u = (-\Delta)^{a/2}u|_{\mathcal{O}} \in L^2(\mathcal{O}).
\end{equation}
Since ${\mathcal{O}^{\prime}} \Subset \mathcal{O}$ implies $\varphi_{a/2}(x)u|_{\mathcal{O}^{\prime}} \in L^2(\mathcal{O}^{\prime})$ and hence we get \eqref{nza}.
\end{proof}

\begin{Lemma}\label{norm0}
Let $\mathcal{O}$ be a bounded Lipschitz domain in $\mathbb{R}^n$. Let  $u\in \widetilde{H}^a(\mathcal{O})$ for $0<a<1$, then $ (-\Delta)^{a/2}_{\mathcal{O}}u\in L^2(\mathcal{O})$ and
\begin{equation}\label{norm_esti}
\|u\|_{L^2(\mathcal{O})}+\|(-\Delta)^{a/2}_{\mathcal{O}}u\|_{L^2(\mathcal{O})}\leq \|u\|_{\widetilde{H}^a(\mathcal{O})}.
\end{equation}
%For $a=\frac{1}{2}$ whenever $u\in H^{1/2}_{0,0}(\mathcal{O})$,  $ (-\Delta)^{1/4}_{\mathcal{O}}u\in L^2(\mathcal{O})$ and 
%\begin{equation}\label{tareq2}\|u\|_{L^2(\mathcal{O})}+\|(-\Delta)^{1/4}_{\mathcal{O}}u\|_{L^2(\mathcal{O})}\leq \|u\|_{H^{1/2}_{0,0}(\mathcal{O})}.\end{equation}
\end{Lemma}
\begin{Remark}
Note that, the above estimate \eqref{norm_esti} holds for $H^{1/2}_{0,0}(\mathcal{O})$ space, whereas, for $u \in H^{1/2}_0(\mathcal{O})$, $(-\D)^{1/4}_{\mathcal{O}} u$ might not be in $L^2(\mathcal{O})$. See the counterexample in  \cite[Section 2]{DYB}.\hfill\qed
\end{Remark}

\begin{proof}[Proof of Lemma \ref{norm0}] 
From  \eqref{ln2} it is enough to show that $\varphi_{a/2}(x)u\in L^2(\mathcal{O})$.
Note that from \eqref{est_phi} we have $\varphi_{a/2}(x)\sim \left(dist\, (x, \mathcal{O}^c)\right)^{-a}$ for $x\in\mathcal{O}$ and $0<a<1$. 
We recall the following fractional Hardy inequalities from \cite{DYB} as follows:
\begin{eqnarray}
&&\mbox{For }0<a<\frac{1}{2}, \quad \forall u\in H^a_0(\mathcal{O}),\notag\\
&&\int_{\mathcal{O}} \frac{|u(x)|^2}{\left(dist\, (x, \mathcal{O}^c)\right)^{2a}}\, dx
\leq 
C\left( \int_{\mathcal{O}}|u|^2\, dx + \int_{\mathcal{O}}\int_{\mathcal{O}} {\frac{(u(x) - u(y))^2}{\left|x-y\right|^{n+2a}} \, dy} \, dx \right),\label{Hardy1}\\
&&\mbox{for }a=\frac{1}{2}, \quad \forall u\in H^{1/2}_{0,0}(\mathcal{O}), \notag\\
&&\int_{\mathcal{O}} \frac{|u(x)|^2}{\left(dist\, (x, \mathcal{O}^c)\right)}\, dx 
\leq 
C\left( \int_{\mathcal{O}}|u|^2\, dx + \int_{\mathcal{O}}\int_{\mathcal{O}} {\frac{(u(x) - u(y))^2}{\left|x-y\right|^{n+1}} \, dy} \, dx \right),\label{Hardy0}\\
&&\mbox{for } \frac{1}{2}<a<1, \quad \forall u\in H^a_0(\mathcal{O}),\notag\\
&&\int_{\mathcal{O}} \frac{|u(x)|^2}{\left(dist\, (x, \mathcal{O}^c)\right)^{2a}}\, dx \leq 
C \int_{\mathcal{O}}\int_{\mathcal{O}} {\frac{(u(x) - u(y))^2}{\left|x-y\right|^{n+2a}} \, dy} \, dx,\label{Hardy2}
\end{eqnarray}
where $C=C(\mathcal{O},n,a)$.

Then for $u\in \widetilde{H}^a(\mathcal{O})$, in all cases, we have that $(-\Delta)^{a/2}_{\mathcal{O}}u\in L^2(\mathcal{O})$. Moreover, from \eqref{ln2} we conclude
\begin{align}
\|(-\Delta)^{a/2}_{\mathcal{O}}u\|_{L^2(\mathcal{O})} &\leq \|\varphi_{a/2}u\|_{L^2(\mathcal{O})} + \|(-\Delta)^{a/2}u\|_{L^2(\mathcal{O})}\notag\\
&\leq C \|u\|_{H^a(\mathcal{O})}  + \|(-\Delta)^{a/2}u\|_{L^2(\mathbb{R}^n)}\label{ineq1}.
\end{align}
We also have that
\begin{align*}
&\|(-\Delta)^{a/2}u\|^2_{L^2(\mathbb{R}^n)} \\
&=  C_{n,a} \int_{\mathbb{R}^n}\int_{\mathbb{R}^n} {\frac{(u(x) - u(y))^2}{\left|x-y\right|^{n+2a}} \, dy} \, dx \\
&=C_{n,a} \int_{\mathcal{O}}\int_{\mathcal{O}} {\frac{(u(x) - u(y))^2}{\left|x-y\right|^{n+2a}} \, dy} \, dx\quad\mbox{(since $u=0$ in $\mathbb{R}^n\setminus\overline{\mathcal{O}}$)}\\
&\qquad+ 2 C_{n,a}\int_{\mathcal{O}}\left[ (u(x))^2 \left(\int_{\mathbb{R}^n\setminus\mathcal{O}} \frac{1}{\left|x-y\right|^{n+2a}} \, dy\right) \right] dx \quad(\mbox{see}  \eqref{est_phi})\\
&\leq C_{n,a} \int_{\mathcal{O}}\int_{\mathcal{O}} {\frac{(u(x) - u(y))^2}{\left|x-y\right|^{n+2a}} \, dy} \, dx + C\int_{\mathcal{O}} \frac{(u(x))^2}{ dist\, \left(x, \mathcal{O}^c\right)^{2a}} \, dx\\
&\leq \|u\|^2_{H^a(\mathcal{O})}.
\end{align*}
Therefore, $\|u\|_{L^2(\mathcal{O})}+ 
\|(-\Delta)^{a/2}_{\mathcal{O}}u\|_{L^2(\mathcal{O})}\leq C\|u\|_{H^a(\mathcal{O})}$.
%\begin{equation}
%\|u\|_{L^2(\mathcal{O})}+ 
%\|(-\Delta)^{a/2}_{\mathcal{O}}u\|_{L^2(\mathcal{O})}\leq C\|u\|_{H^a(\mathcal{O})}.
%\end{equation}
%This completes the proof of the lemma. 
\hfill\end{proof}

\begin{Lemma}\label{rev3}
Let $u\in \widetilde{H}^r(\mathcal{O})$ for $0<a\leq r$. Then $(-\Delta)^{a/2}_\mathcal{O}u\in H^{r-a}(\mathcal{O})$.
\end{Lemma}
\begin{proof}
Let us write 
\begin{equation}\label{rev5}
(-\Delta)^{a/2}_{\mathcal{O}}u  = (-\Delta)^{a/2}u-\varphi_{a/2}u
\end{equation}
for $u\in \widetilde{H}^r(\mathcal{O})$.
Thanks to \eqref{mapFL}, we already know that $\left((-\Delta)^{a/2}u\right)|_\mathcal{O}\in H^{r-a}(\mathcal{O})$.
Moreover, from \cite{Grubb1} we have that $\frac{u}{d(x,\mathcal{O}^c)^a}\in H^{r-a}(\mathcal{O})$ for $u\in \widetilde{H}^r(\mathcal{O})$. Since $\varphi_{a/2}\sim d(x,\mathcal{O}^c)^{-a}$, therefore from \eqref{rev5} it follows $(-\Delta)^{a/2}_{\mathcal{O}}u\in H^{r-a}(\mathcal{O})$.
\end{proof}

Let $0<a<\min\{1,\frac{n}{2}\}$, the well-known Hardy-Littlewood-Sobolev inequality follows as (see \cite[Proposition 15.5]{PAC}) 
\begin{equation}
\label{hls} \|u\|_{L^{\frac{2n}{n-2a}}(\mathbb{R}^n)}\leq C \|(-\Delta)^{a/2}u\|_{L^2(\mathbb{R}^n)},\quad\forall u\in C^\infty_c(\mathbb{R}^n)
\end{equation}
where $C$ depends on $n$ and $a$.
Let $\mathcal{O}\subset \mathbb{R}^n$ be a bounded Lipschitz domain. For $u\in \widetilde{H}^a(\mathcal{O})$ with $0<a<\min \{1, \frac{n}{2}\}$ we have the following inequality
\begin{equation}\label{Hardy_Sobolev}
\|u\|_{L^2(\mathcal{O})}\leq C_{\mathcal{O}}\|u\|_{L^{\frac{2n}{n-2a}}(\mathbb{R}^n)}\leq C \|(-\Delta)^{a/2}u\|_{L^2(\mathbb{R}^n)}.
\end{equation}
%as $supp\, u\subseteq \overline{\mathcal{O}}$.
\begin{Proposition}[\cite{HRAV, DIV}]
For $0<a<1$ we have the following Poincar\'{e}-Wirtinger inequality
\begin{equation}\label{Poincare1} 
\left\lVert u-\frac{1}{|\mathcal{O}|}\int_\mathcal{O}u\right\rVert_{L^2(\mathcal{O})} \leq 
C \left(\int_{\mathcal{O}}\int_{\mathcal{O}} {\frac{(u(x) - u(y))^2}{\left|x-y\right|^{n+2a}} \, dy} \, dx\right)^{\frac{1}{2}}, \quad
\forall u\in H^a(\mathcal{O});
\end{equation}
where $C=C(\mathcal{O},n,a)$.
\end{Proposition}
%\noindent See \cite{HRAV, DIV} for the proof of the above proposition.

\begin{Lemma}[Poincar\'e inequality]
Let $\mathcal{O} \subset \R^n$, $n\geq 1$ be a bounded open set with Lipschitz boundary and $u\in \widetilde{H}^a(\mathcal{O})$, then for some $C=C(\mathcal{O},n,a)$ we have
\begin{equation}\label{Poincare2} 
%\forall u \in \widetilde{H}^a(\mathcal{O}),\quad 
\|u\|_{L^2(\mathcal{O})} \leq C \left(\int_{\mathcal{O}}\int_{\mathcal{O}} {\frac{(u(x) - u(y))^2}{\left|x-y\right|^{n+2a}} \, dy} \, dx\right)^{\frac{1}{2}}
\mbox{ whenever }\frac{1}{2}<a<1.
\end{equation}
\end{Lemma}
\begin{proof}
We have for $u\in \widetilde{H}^a(\mathcal{O})$
\[ \|(-\Delta)^{a/2}u\|^2_{L^2(\mathbb{R}^n)} =  \int_{\mathcal{O}}\int_{\mathcal{O}} {\frac{(u(x) - u(y))^2}{\left|x-y\right|^{n+2a}} \, dy} \, dx +  2\int_{\mathcal{O}}\int_{\mathbb{R}^n\setminus\mathcal{O}} {\frac{(u(x) - u(y))^2}{\left|x-y\right|^{n+2a}} \, dy} \, dx.\]
Using \eqref{est_phi} and \eqref{Hardy2} on the second term of the above identity for $\frac{1}{2}<a<1$, we get \eqref{Poincare2}. 
\end{proof}

\begin{Remark}
The above inequality is not true for $0<a\leq \frac{1}{2}$ (see \cite{DYB}). For example, the characteristic function $\chi_\mathcal{O}\in \widetilde{H}^a(\mathcal{O})$ for $a\in (0,\frac{1}{2})$, and \eqref{Poincare2} does not hold in this case. For $a=\frac{1}{2}$ see the Remark \ref{prerem}.
%In general, we have the following Poincar\'{e}-Wirtinger inequality (see \cite{HRAV, DIV}) as
%\begin{align}\label{Poincare1} 
%\forall u\in H^a(\mathcal{O}), \quad
%&\left\lVert u-\frac{1}{|\mathcal{O}|}\int_\mathcal{O}u\right\rVert_{L^2(\mathcal{O})} \leq 
%C \left(\int_{\mathcal{O}}\int_{\mathcal{O}} {\frac{(u(x) - u(y))^2}{\left|x-y\right|^{n+2a}} \, dy} \, dx\right)^{\frac{1}{2}}\\
%&\qquad\qquad\qquad\qquad\qquad\qquad\qquad\qquad\qquad\mbox{whenever $0<a<1$}\notag
%\end{align}
%where $C=C(\mathcal{O},n,a)$.
\end{Remark}

\subsection{Direct problem for the regional fractional Laplacian operator}\label{d2}
In this article we discuss two direct problems corresponding to the non-local operators $(-\D)^a_{\Omega}$ and $\mL_{b,q}$.
Let us start with the regional fractional Laplacian operator $(-\D)^a_{\Omega}$.
Let $0<a<1$, $\mathcal{O} \Subset \Omega$, i.e. $\mathcal{O}$ is compactly contained in $\Omega$ and $f \in H^{-a}(\mathcal{O})$, $g \in H^a(\Omega\setminus \overline{\mathcal{O}})$, we solve  the Dirichlet problem
\begin{equation}\label{Dir_Prob_RFL}
(-\D)^a_{\Omega} u = f \quad \mbox{in } \mathcal{O}, \quad \mbox{and }u = g \quad \mbox{on }\Omega\setminus\overline{\mathcal{O}}.
\end{equation} 
We define the corresponding bilinear form $\mathcal{B}_{\Omega}: H^a(\Omega)\times H^a(\Omega)\mapsto \mathbb{R}$ by
\begin{equation*}
\mathcal{B}_{\Omega}(\varphi,\psi) := \int_{\Omega}\int_{\Omega}{\frac{(\varphi(x) - \varphi(y))(\psi(x) - \psi(y))}{\left|x-y\right|^{n+2a}} \, dy} \, dx.
\end{equation*}
We say $u \in H^a(\Omega)$ is a weak solution of \eqref{Dir_Prob_RFL} if $\mathcal{B}_{\Omega}(u,\psi) = \langle f, \psi \rangle$, for all $\psi \in \tilde{H}^a(\mathcal{O})$ with 
$u=g$ on $\Omega\setminus \overline{\mathcal{O}}$.

\begin{Theorem}\label{existence_th2}
Let $\Omega \subset \R^n$, bounded Lipschitz domain, and  $\mathcal{O}\subset \Omega$ be a non-empty open Lipschitz domain compactly contained in $\Omega$.  
Let $f \in H^{-a}(\mathcal{O})$ and $G \in H^{a}(\Omega)$ with $G=g$ in $\Omega\setminus\overline{\mathcal{O}}$, $0<a<1$. There exists a unique weak solution $v \in H^a(\Omega)$ solving \eqref{Dir_Prob_RFL}. Moreover, it satisfies the
following stability estimate
\begin{equation}\label{stability}
\lVert v \rVert_{H^a(\Omega)} 
\leq C\left( \lVert f \rVert_{H^{-a}(\mathcal{O})} 
+ \lVert G \rVert_{H^a(\Omega)} \right).
\end{equation}
\end{Theorem}
%\noindent We prove the above theorem in Section \ref{d2}.
\subsection*{Proof of Theorem \ref{existence_th2}:}\label{d3}
%Let $\Omega$ be a bounded Lipschitz domain in $\mathbb{R}^n$, and  $\mathcal{O}\Subset \Omega$ be an open subset {\color{blue}with Lipschitz %boundary}, compactly contained in $\Omega$. We consider the following non-local problem
%\begin{equation}
%\begin{aligned}\label{inhomeq2}
%(-\Delta)^a_{\Omega}u &= f \quad\mbox{ in }\mathcal{O}\\
%u &= g\quad\mbox{in }\Omega\setminus\mathcal{O},
%\end{aligned}
%\end{equation}
%
%where $f \in H^{-a}(\mathcal{O})$ and $g \in H^a(\Omega\setminus \overline{\mathcal{O}})$.
\textbf{Homogeneous Case:}
Let us begin with the homogeneous boundary value problem, i.e. when $g=0$ in $\Omega\setminus\overline{\mathcal{O}}$.
Let $f\in H^{-a}(\mathcal{O})$, we say $v_f\in \widetilde{H}^a(\mathcal{O})$, $0<a<1$, is the weak solution of
\begin{equation}\label{weaksl2}
(-\Delta)^a_{\Omega} v = f \quad \mbox{in }\mathcal{O}, \quad v = 0 \quad \mbox{in }\Omega\setminus\overline{\mathcal{O}},
\end{equation}
if for all $w\in C^\infty_c(\mathcal{O})$
\begin{equation}\label{nbil2}
\mathcal{B}_{\Omega}(v_f,w)= \langle f, w\rangle_{(H^{-a}(\mathcal{O}),\widetilde{H}^a(\mathcal{O}))}.
\end{equation}
%holds, where the bilinear form $\mathcal{B}_{\Omega}: H^a_0(\Omega)\times H^a_0(\Omega)\mapsto \mathbb{R}$ is defined as
%\begin{equation*}
%\mathcal{B}_{\Omega}(\varphi,\psi) := \int_{\Omega}\int_{\Omega}{\frac{(\varphi(x) - \varphi(y))(\psi(x) - \psi(y))}{\left|x-y\right|^{n+2a}} \, dy} \, dx.
%\end{equation*}
A straightforward calculation shows that the bilinear form $\mathcal{B}(\cdot,\cdot)$ is continuous over $\widetilde{H}^a(\Omega)\times \widetilde{H}^a(\Omega)$.
Next we show that $\mathcal{B}_{\Omega}(\cdot,\cdot)$ is coercive over  $\widetilde{H}^a(\mathcal{O})$ for $\mathcal{O}\Subset\Omega$ i.e.
\begin{equation}\label{poin2}
\mathcal{B}_{\Omega}(\varphi,\varphi) = \int_{\Omega}\int_{\Omega}{\frac{(\varphi(x) - \varphi(y))^2}{\left|x-y\right|^{n+2a}} \, dy} \, dx \geq C \|\varphi\|^2_{\widetilde{H}^a(\mathcal{O})},\quad\forall \varphi\in \widetilde{H}^a(\mathcal{O}).
\end{equation}
Let $\varphi\in C^\infty_c(\mathcal{O})$, using the Poincar\'{e}-Wirtinger inequality \eqref{Poincare1} over $\Omega$, we get
\begin{equation*}
\begin{aligned}
C \left(\int_{\Omega}\int_{\Omega} {\frac{(\varphi(x) - \varphi(y))^2}{\left|x-y\right|^{n+2a}} \, dy} \, dx\right)^{\frac{1}{2}} 
&\geq \left\lVert\varphi-\frac{1}{|\Omega|}\int_\Omega \varphi\right\rVert_{L^2(\Omega)} \\
& \geq \lVert\varphi\rVert_{L^2(\mathcal{O})}-\frac{1}{|\Omega|^{\frac{1}{2}}}\int_\Omega |\varphi| \\
& \geq \|\varphi\|_{L^2(\mathcal{O})}-\frac{1}{|\Omega|^{\frac{1}{2}}}\int_\Omega\chi_\mathcal{O}\, |\varphi|\\
&\geq \left( 1-\frac{|\mathcal{O}|^{\frac{1}{2}}}{|\Omega|^{\frac{1}{2}}}\right)\|\varphi\|_{L^2(\mathcal{O})}.
\end{aligned}
\end{equation*}
Hence, we have \eqref{poin2}. 

Therefore, using the Lax-Milligram theorem, for a given $f\in H^{-a}(\mathcal{O})$, we have a unique weak solution of \eqref{weaksl2} in $\widetilde{H}^a(\mathcal{O})$ satisfying the stability estimate 
\[ \|v_f\|_{\widetilde{H}^a(\mathcal{O})}\leq \|f\|_{H^{-a}(\mathcal{O})}.\]

\begin{Remark}
	Here we remark that, in establishing the coercivity of  the bilinear form  $\mathcal{B}_\Omega(\cdot,\cdot)$ over the $\widetilde{H}^a(\mathcal{O})$ space, the assumption $\mathcal{O}\Subset\Omega$ is crucial.  For example, $\mathcal{B}_\Omega(\cdot,\cdot)$ fails to become coercive over the space $\widetilde{H}^a(\Omega)$ while $0<a<\frac{1}{2}$.
\end{Remark}
\noindent
\textbf{Inhomogeneous Case:} Let $G\in H^{a}(\Omega)$ then from \eqref{acon} we know $(-\Delta)^a_{\Omega}G\in H^{-a}(\Omega)$, $0<a<1$. Now we are interested in   the following inhomogeneous problem 
\begin{equation*} 
(-\Delta)^a_{\Omega} v = f \quad\mbox{in }\mathcal{O},\quad (v-G)\in \widetilde{H}^{a}(\mathcal{O}).
\end{equation*}
Then $w=(v-G)\in \widetilde{H}^{a}(\mathcal{O})$ solves 
\[ (-\Delta)^a_{\Omega}w= f-(-\Delta)^a_{\Omega}G\in H^{-a}(\mathcal{O}),\]
and by the previous discussion we have a unique weak solution in $\widetilde{H}^{a}(\mathcal{O})$,
where as the stability estimate \eqref{stability} follows from Proposition \ref{Prop_cont}.
\qed

\begin{Corollary}\label{cr2}
The operator 
\[\left((-\Delta)^{a}_{\Omega}\right)^{-1}: H^{-a}(\mathcal{O})\rightarrow \widetilde{H}^{a}(\mathcal{O})\]
is one-one, onto and bounded.  
\end{Corollary}

\subsection{Direct problem for $\mL_{b,q}$}\label{Sec_direct_prob}
Here we study the direct problem for the operator $\mL_{b,q}$.
Let $\Omega$ be an open, bounded Lipschitz domain in $\mathbb{R}^n$. We consider the inhomogeneous problem
\begin{equation}\label{eq2}
\begin{aligned}
\mL_{b,q} u:= \left((-\D)^t + (-\D)_{\Omega}^{s/2}b(-\D)_{\Omega}^{s/2} + q\right)u &= F  \quad &&\mbox{in }\Omega,\\
u &= f \quad &&\mbox{in }\Omega_e,
\end{aligned}
\end{equation}
where $F \in H^{-t}(\Omega)$, $f \in \widetilde{H}^t(\Omega_e)$ and $b,\, q\in L^\infty(\Omega)$.\\
\\
The bilinear form associated to the operator $\mL_{b,q}$ is
\begin{equation*}
\mathcal{B}_{b,q}: H^t(\R^n)\times H^t(\R^n)\rightarrow \R
\end{equation*}
given by 
\begin{equation}\label{bilinear}
\begin{aligned} 
\mathcal{B}_{b,q}(\varphi,\psi) &:= \int_{\R^n} (-\D)^{t/2}\varphi(x)\, (-\D)^{t/2}\psi(x) \, dx \\
&\quad+ \int_{\Omega}b(x)\left((-\D)_{\Omega}^{s/2}\varphi\right)(x)\ \left((-\D)_{\Omega}^{s/2}\psi\right)(x)\, dx + \int_{\Omega} q(x)\varphi(x)\,\psi(x)\, dx.
\end{aligned}
\end{equation}
We define $u \in H^t(\R^n)$ to be a weak solution of \eqref{eq2} if for every $\varphi \in C^\infty_c(\Omega)$ we have
\[ \mathcal{B}_{b,q}(u,\varphi) = \langle F, \varphi\rangle, \quad \mbox{with }u=f \mbox{ in }\Omega_e.\]
We state the following theorem for wellposedness of the Dirichlet problem \eqref{eq2}

\begin{Theorem}\label{existence_th1}
Let $\Omega \subset \R^n$, $n\geq 1$, be a Lipschitz domain and $b,q \in L^{\infty}(\Omega)$.
Let $0<t<1$, for any $F \in H^{-t}(\Omega)$ and $f \in \widetilde{H}^t(\Omega_e)$ there exist a unique weak solution $u \in H^t(\R^n)$ solving the Dirichlet problem \eqref{eq2}.
It satisfies the following stability estimate
\begin{equation*}
\lVert u \rVert_{H^t(\R^n)} 
\leq C\left( \lVert F \rVert_{H^{-t}(\Omega)} 
+ \lVert f \rVert_{H^t(\Omega_e)} \right).
\end{equation*} 

\end{Theorem}

\begin{proof}
By extending $f\in \widetilde{H}^t(\Omega_e)$ by $0$ in $\Omega$ as a $H^t(\mathbb{R}^n)$ function, we observe that the theorem is equivalent to considering the following homogeneous problem for $v=(u-f)\in \widetilde{H}^t(\Omega)$ and $\widetilde{F}:=F-(-\Delta)^tf$,
\begin{equation}\begin{aligned}\label{eq3}
\mL_{b,q} v &=\widetilde{F}  \quad \mbox{in }\Omega,\\
v&=0 \quad\mbox{in }\Omega_e.
\end{aligned}\end{equation}
Equivalently, in terms of the bilinear form we seek $v\in \widetilde{H}^t(\Omega)$ solving
\begin{equation}\label{rev1}
\mathcal{B}_{b,q}(v,\varphi) = \langle \widetilde{F}, \varphi\rangle, \quad\forall \varphi \in \widetilde{H}^t(\Omega),
\end{equation}
where $\langle\cdot,\cdot\rangle$ denotes the usual duality between the spaces  $H^{-t}(\Omega)$ and $\widetilde{H}^t(\Omega)$. Let us note that,
\[ \langle \widetilde{F}, \varphi\rangle_{H^{-t}(\Omega), \widetilde{H}^t(\Omega)}\leq \left(\|F\|_{H^{-t}(\Omega)} + \|f\|_{H^t(\Omega_e)}\right)\,\|\varphi\|_{\widetilde{H}^t(\Omega)}.\]
In order to prove the existence of the solution of \eqref{eq2}, now, we will show the existence of a solution $v \in \widetilde{H}^t(\Omega)$ solving \eqref{eq3}.

\subsection*{Continuity of the bilinear form $\mathcal{B}_{b,q}(\cdot,\cdot)$}  
Let $\varphi, \psi \in \widetilde{H}^t(\Omega)$, first note that due to \eqref{nza} for $0<s<t<1$ with $s\neq \frac{1}{2}$ we have
\[||(-\D)^{s/2}_{\Omega}\varphi||_{L^2(\Omega)} \leq C || \varphi ||_{\widetilde{H}^s(\Omega)} \leq || \varphi ||_{\widetilde{H}^t(\Omega)}.\]

For $s=\frac{1}{2}$ and $\varphi\in\widetilde{H}^t(\Omega)$ with $\frac{1}{2}<t<1$, we find that $\varphi\in H^{1/2}_{0,0}(\Omega)$ satisfying 
\[||(-\D)^{1/4}_{\Omega}\varphi||_{L^2(\Omega)} \leq C || \varphi ||_{H^{1/2}_{0,0}(\Omega)} \leq || \varphi ||_{\widetilde{H}^t(\Omega)}.\]

To see this, let $\varphi\in \widetilde{H}^r(\Omega)$ for some $\frac{1}{2}<r<1$, we claim that $\varphi\in H^{1/2}_{0,0}(\Omega)$, satisfying $||(-\D)^{1/4}_{\Omega}\varphi||_{L^2(\Omega)} \leq C || \varphi ||_{H^{1/2}_{0,0}(\Omega)} \leq || \varphi ||_{\widetilde{H}^r(\Omega)}$. 
By Using \eqref{tareq} we have $\frac{\varphi}{d(x,\Omega^c)^r}\in L^2(\Omega)$, where $d(x, \Omega^c)$ is a smooth positive extension into $\Omega$ of $dist(x,\Omega^c)$ near $\partial\Omega$.
Now since $\frac{1}{2}<r<1$, we have that$\frac{\varphi}{d(x,\Omega^c)^{\frac{1}{2}}}\in L^2(\Omega)$. Hence $\varphi\in H^{1/2}_{0,0}(\Omega)$ and using \eqref{norm_esti} the above estimate follows.

Therefore, we have
\begin{equation}\label{cont_bilin}
\begin{aligned}
\lvert \mathcal{B}_{b,q}(\varphi,\psi)\rvert 
&\leq || (-\D)^{t/2}\varphi||_{L^2(\R^n)} || (-\D)^{t/2}\psi||_{L^2(\R^n)}\\ &\quad+||b||_{L^\infty(\Omega)}||(-\D)^{s/2}_{\Omega}\varphi||_{L^2(\Omega)}||(-\D)^{s/2}_{\Omega}\psi||_{L^2(\Omega)}\\ &\quad+||q||_{L^\infty(\Omega)}||\varphi||_{L^2(\Omega)}||\psi||_{L^2(\Omega)},\\
&\leq C ||\varphi||_{\widetilde{H}^t(\Omega)}||\psi||_{\widetilde{H}^t(\Omega)}.
\end{aligned}
\end{equation}
\noindent{\bf Coercivity of the bilinear form $\mathcal{B}_{b,q}$ on $\widetilde{H}^t(\Omega)$.}
Let $0<s<t<1$ and $\varphi \in \widetilde{H}^t(\Omega)$. Then
%using the Hardy-Littlewood-Sobolev inequality (see \eqref{Hardy_Sobolev}) on a bounded domain
%\[
%\lVert \varphi \rVert_{L^2(\Omega)} \leq C\lVert \varphi \rVert_{L^{\frac{2n}{n-2t}}(\Omega)} \leq C\lVert (-\Delta)^{t/2}\varphi \rVert_{L^2(\R^n)}
%\]
for $\sigma>\lVert q\rVert_{L^{\infty}(\Omega)}$ we obtain
\begin{equation}\label{Coercivity}
\begin{aligned}
\mathcal{B}_{b,q}(\varphi,\varphi) + \sigma\langle \varphi,\varphi\rangle_{L^2(\Omega)}
\geq&\, || (-\D)^{t/2}\varphi||^2_{L^2(\R^n)} + \sigma \lVert \varphi \rVert^2_{L^2(\Omega)}\\
&-||b||_{L^\infty(\Omega)}||(-\D)^{s/2}_{\Omega}\varphi||^2_{L^2(\Omega)}- ||q||_{L^\infty(\Omega)}||\varphi||^2_{L^2(\Omega)}\\
\geq&\, || (-\D)^{t/2}\varphi||^2_{L^2(\R^n)} + (\sigma - \lVert q \rVert_{L^{\infty}(\Omega)}) \lVert \varphi \rVert^2_{L^2(\Omega)}\\
&-||b||_{L^\infty(\Omega)}||(-\D)^{s/2}_{\Omega}\varphi||^2_{L^2(\Omega)}\\
\geq&\,\tilde{C}||\varphi||^2_{\widetilde{H}^t(\Omega)}-||b||_{L^\infty(\Omega)}||\varphi||^2_{\widetilde{H}^r(\Omega)}
\end{aligned}
\end{equation}
where $r=s$ when $s\neq \frac{1}{2}$, and for $s=\frac{1}{2}$, we take some fixed $r\in (\frac{1}{2},t)$.

	\noindent Given the compact inclusions
\[ \widetilde{H}^t(\Omega)\hookrightarrow \widetilde{H}^r(\Omega)
\hookrightarrow L^2(\Omega), \quad\mbox{for }0<r<t<1, \]
we have for $\lambda>0$ (see \cite[Lemma 2.1]{Temam})
\[ ||\varphi||^2_{\widetilde{H}^r(\Omega)}\leq \frac{1}{\lambda}||\varphi||^2_{\widetilde{H}^t(\Omega)} + C_{\lambda} ||\varphi||^2_{L^2(\Omega)}, \quad\mbox{ for }\varphi\in \widetilde{H}^t(\Omega).\]
Therefore, combining the above estimates we get for $\lambda \geq 2\tilde{C}^{-1}\lVert b\rVert_{L^{\infty}(\Omega)}$
\begin{equation}\label{coerc_bilin}
\mathcal{B}_{b,q}(\varphi,\varphi) +(\sigma+C_{\lambda})\lVert\varphi\rVert^2_{L^2(\Omega)} \geq \frac{\tilde{C}}{2}\lVert\varphi\rVert^2_{\widetilde{H}^t(\Omega)}, \quad\mbox{ for } \varphi\in \widetilde{H}^t(\Omega).
\end{equation}
By the Riesz-representation theorem there exists a unique $w = G_{\mu} (\widetilde{F}) \in \widetilde{H}^t(\Omega)$,
where $G_{\mu}:H^{-t}(\Omega)\to\widetilde{H}^t(\Omega)$ be a bounded map,  such that
\begin{equation*}
\mathcal{B}_{b,q}(w,\varphi) + \mu\langle w,\varphi\rangle = \langle \widetilde{F},\varphi\rangle, \qquad \forall\varphi \in \widetilde{H}^t(\Omega), \quad \mu\geq (\sigma + C_{\lambda}).
\end{equation*}
Hence, we have unique $w \in \widetilde{H}^t(\Omega)$ satisfying
\begin{equation*}
\left(\mL_{b,q} + \mu I\right) w = \widetilde{F}, \quad \mbox{in }\Omega,
\end{equation*}
where $I : \widetilde{H}^t(\Omega) \to H^{-t}(\Omega)$ is the identity map.

\noindent
Observe that, $v \in \widetilde{H}^t(\Omega)$ satisfies $\mL_{b,q} v =\widetilde{F}$ if and only if
\begin{equation*}
\left(\mL_{b,q} + \mu I\right)v -\mu I v = \widetilde{F} \quad \iff \quad
v - \mu (G_{\mu}\circ I) v = G_{\mu}(\widetilde{F}).
\end{equation*}
%
%
%and it follows
%\[ \mathcal{B}_{b,q}(v,\varphi) - \theta\langle v, \varphi \rangle = \langle\widetilde{F},\varphi\rangle, \quad 
%\mbox{for } v = G\left((\mu + \theta)v + \widetilde{F},\mu\right), \quad \forall\theta \in \R. \]
%Since the bilinear form $\mathcal{B}_{b,q}(\cdot,\cdot)$ is symmetric, thus $G(\cdot,\mu)$ is  self-adjoint.
Now, using the Rellich-Kondrachov compact embedding theorem we have
\begin{equation*}
I : \widetilde{H}^t(\Omega) \to H^{-t}(\Omega)
\end{equation*}
is compact and consequently
%\begin{equation}\label{G}
$\left(G_{\mu} \circ I\right): \widetilde{H}^t(\Omega) \to \widetilde{H}^t(\Omega)$
%\end{equation}
is a compact operator. 
%
%
%
%Therefore, from the standard spectral analysis $G(\cdot,\mu)$ has a discreet spectrum $\Sigma \subset \R$.
%That is for $\theta \notin \Sigma$, there is a unique solution $\tilde{v} = \tilde{v}(\cdot,\theta) \in \widetilde{H}^t(\Omega)$ solving 
%\begin{equation}\label{key_1}
%\mL_{b,q} \tilde{v} - \theta \tilde{v} = \widetilde{F}  \quad \mbox{in }\Omega, \quad \tilde{v} \in \widetilde{H}^t(\Omega), \quad \theta \in \R \setminus \Sigma.
%\end{equation}
Hence, by the Fredholm alternative theorem, existence of a solution $v \in \widetilde{H}^t(\Omega)$ of $v - \mu (G_{\mu}\circ I) v = G_{\mu}(\widetilde{F})$ in $\Omega$ follows from the uniqueness of the trivial solution of $\mu (G_{\mu}\circ I) v = v$ in $\Omega$ in the function space $\widetilde{H}^t(\Omega)$ (c.f. Assumption \eqref{eva}).
%
%Hence, by the Fredholm alternative theorem, existence of a solution $v \in \widetilde{H}^t(\Omega)$ of $\mL_{b,q}v=\widetilde{F}$ in $\Omega$ follows from uniqueness in $\widetilde{H}^t(\Omega)$ of the trivial solution of $\mL_{b,q} u = 0$ in $\Omega$ (c.f. Assumption \eqref{eva}).
Therefore, we get a unique solution $v\in \widetilde{H}^t(\Omega)$ of $\mL_{b,q}v=\widetilde{F}$, in $\Omega$, and consequently a unique solution $u=(v+f)\in H^t(\R^n)$ solving
\begin{equation*}
\begin{aligned}
\mL_{b,q} u = F \quad &\mbox{in }\Omega,\\
u = f \quad &\mbox{in }\Omega_e,
\end{aligned}
\end{equation*}
for $F=(\widetilde{F} + (-\Delta)^t f) \in H^{-t}(\Omega)$ and $f \in \widetilde{H}^t(\Omega_e)$.

\subsection*{Stability estimate} 
Now we will show that if $u$ is the unique solution of \eqref{eq2} then the following stability estimate is true:
\[
\lVert u\rVert_{H^t(\R^n)} \leq C\left(\lVert F\rVert_{H^{-t}(\Omega)} + \lVert f\rVert_{\widetilde{H}^t(\Omega_e)}\right).
\]
In order to show that, observe that from \eqref{coerc_bilin} we get
\[
\lVert v \rVert^2_{\widetilde{H}^t(\Omega)} \leq C\lVert v \rVert^2_{L^2(\Omega)} + \mathcal{B}_{b,q}(v,v).
\]
Now as $v$ solves the \eqref{eq3} we get $\mathcal{B}_{b,q}(v,v) = \langle\widetilde{F},v\rangle$. 
Hence,
\begin{align*}\lVert v \rVert^2_{\widetilde{H}^t(\Omega)} &\leq C\lVert v \rVert^2_{L^2(\Omega)} + \rvert\langle \widetilde{F},v\rangle\lvert\\
&\leq C\left(\lVert v \rVert_{L^2(\Omega)} + \|\widetilde{F}\|_{H^{-t}(\Omega)}\right)\|v\|_{\widetilde{H}^t(\Omega)}
\end{align*}
or,
\begin{equation*}
\lVert v \rVert_{\widetilde{H}^t(\Omega)} \leq C\left(\lVert v \rVert_{L^2(\Omega)} + \|\widetilde{F}\|_{H^{-t}(\Omega)}\right).
\end{equation*}
Now by putting $u=v+f$ and $\widetilde{F} = F - (-\Delta)^{t}f$ with $\|\widetilde{F}\|_{H^{-t}(\Omega)}\leq \|{F}\|_{H^{-t}(\Omega)} +\|f\|_{\widetilde{H}^t(\Omega_e)}$ one gets 
\[
\lVert u\rVert_{H^t(\R^n)} \leq C\left( \|u\|_{L^2(\Omega)}+ \lVert F\rVert_{H^{-t}(\Omega)} + \lVert f\rVert_{\widetilde{H}^t(\Omega_e)}\right).
\]
Moreover, by using the compactness of the inverse operator $(G_{\mu}\circ I)$ we obtain 
\[
\lVert u\rVert_{H^t(\R^n)} \leq C\left(\lVert F\rVert_{H^{-t}(\Omega)} + \lVert f\rVert_{\widetilde{H}^t(\Omega_e)}\right).
\]
Hence, the stability estimate follows.
\end{proof}

\begin{Remark}
In this remark, we will discuss the case where $f\in H^t(\Omega_e)$ instead of $f\in \widetilde{H}^t(\Omega_e)$ in \eqref{eq2}. In this case, we also assume $b\in H^1(\Omega)\cap L^\infty(\Omega)$, and we will see how the higher regularity of $b$ helps to solve the forward problem \eqref{eq2} with more general exterior data $f\in H^t(\Omega_e)$.
Let us call $\widetilde{f}\in H^t(\mathbb{R}^n)$ be a non-zero extension of $f\in H^t(\Omega_e)$, i.e. $\widetilde{f}|_{\Omega_e}=f$  satisfying $\|\widetilde{f}\|_{H^t(\mathbb{R}^n)}\leq C\|f\|_{H^t(\Omega_e)}$.
%Note that, $\widetilde{f}\neq 0$ in $\overline{\Omega}$.
Now let us consider the modified (w.r.t. \eqref{rev1}) homogeneous problem  in this case, as $\widetilde{v} =u-\widetilde{f}\in H^t(\mathbb{R}^n)$ solves 
\begin{equation}\label{rev4}
\mathcal{B}_{b,q}(\widetilde{v},\varphi) = \langle F-(-\Delta)^t\widetilde{f}, \varphi\rangle- \langle b(-\Delta)^{s/2}_\Omega \widetilde{f},(-\Delta)^{s/2}_\Omega\varphi\rangle- \langle q\widetilde{f},\varphi\rangle, \quad\forall \varphi \in \widetilde{H}^t(\Omega)
\end{equation}
Let's consider the r.h.s. of \eqref{rev4}.  
It's easy to observe that %the third and the first terms satisfy
\begin{equation*}
\begin{aligned}
&\langle q\widetilde{f},\varphi\rangle_{L^2(\Omega)} 
\leq \|q\|_{L^\infty(\Omega)}\|f\|_{L^2(\Omega_e)}\|\varphi\|_{L^2(\Omega)}\\
\mbox{and}\qquad 
&\langle F-(-\Delta)^t\widetilde{f}, \varphi\rangle_{H^{-t}(\Omega), \widetilde{H}^t(\Omega)}
\leq \left(\|F\|_{H^{-t}(\Omega)} + \|f\|_{H^t(\Omega_e)}\right)\,\|\varphi\|_{\widetilde{H}^t(\Omega)}.
\end{aligned}
\end{equation*}
In the second term, we will use $H^r$-duality instead of the  $L^2$-inner product as done earlier.
We get that, due to \eqref{acon2} $(-\Delta)^{s/2}_\Omega(\widetilde{f}|_\Omega)\in H^{-s/2+\delta}(\Omega)$, for any $0\leq \delta< \frac{s}{2}$ and $\widetilde{f}|_{\Omega}\in H^{t}(\Omega)$ such that $\frac{s}{2}+\delta\leq t<1$. %since $0<s<t<1$.
On the other hand, $(-\Delta)^{s/2}_\Omega\varphi \in H^{t-s}(\Omega)$ for $\varphi\in \widetilde{H}^{t}(\Omega)$.
By choosing $\delta\in [0,\frac{s}{2})$ close to $\frac{s}{2}$, for given $0<s<t<1$ it is always possible to have some $\delta=\delta_0$ such that $t-s\geq \frac{s}{2}-\delta_0$. Therefore, for that $\delta_0\in [0,\frac{s}{2})$ we have $(-\Delta)^{s/2}_\Omega\varphi \in H^{s/2-\delta_0}(\Omega)$. Now by using the regularity of $b\in H^1(\Omega)$, we have $b(-\Delta)^{s/2}_\Omega\varphi\in H^{s/2-\delta_0}(\Omega)$ as well. Thus 
\[\begin{aligned}
&\langle (-\Delta)^{s/2}_\Omega \widetilde{f}, b(-\Delta)^{s/2}_\Omega\varphi\rangle_{H^{-s/2+\delta_0}(\Omega),\, H^{s/2-\delta_0}(\Omega)}\\
&\leq \|b\|_{H^1(\Omega)}\|(-\Delta)^{s/2}_\Omega\widetilde{f}\|_{H^{-s/2+\delta_0}(\Omega)}\|(-\Delta)^{s/2}_\Omega\varphi\|_{H^{s/2-\delta_0}(\Omega)}.
\end{aligned}  \]
Again by using \eqref{acon2}, we have $\|(-\Delta)^{s/2}_\Omega\widetilde{f}\|_{H^{-s/2+\delta_0}(\Omega)}\leq \|\widetilde{f}\|_{H^{s/2+\delta_0}(\Omega)}\leq \|\widetilde{f}\|_{H^{t}(\Omega)}$, since $\frac{s}{2}+\delta_0<s<t$. By Lemma \ref{rev3}, $\|(-\Delta)^{s/2}_\Omega\varphi\|_{H^{s/2-\delta_0}(\Omega)} \leq C\|\varphi\|_{\widetilde{H}^{3s/2-\delta_0}(\Omega)}\leq C\|\varphi\|_{\widetilde{H}^{t}(\Omega)}$, since by our choice  $t\geq \frac{3s}{2}-\delta_0$. Hence, 
\[\langle (-\Delta)^{s/2}_\Omega \widetilde{f}, b(-\Delta)^{s/2}_\Omega\varphi\rangle_{H^{-s/2+\delta_0}(\Omega),\, H^{s/2-\delta_0}(\Omega)}\leq C\|b\|_{H^1(\Omega)}\|\widetilde{f}\|_{H^{t}(\Omega)}\|\varphi\|_{H^{t}(\Omega)}.\]
So from \eqref{rev4} we get 
$|\mathcal{B}_{b,q}(v,\varphi)| \leq  \left(\|F\|_{H^{-t}(\Omega)} + \|f\|_{H^t(\Omega_e)}\right)\,\|\varphi\|_{\widetilde{H}^t(\Omega)}$.

Then following the same procedure as we did before, one can solve the homogeneous problem \eqref{rev4} in $\widetilde{H}^t(\Omega)$, and consequently the inhomogeneous problem \eqref{eq2} in $H^t(\mathbb{R}^n)$ with the exterior data in $H^t(\Omega_e)$.
\end{Remark}

\subsection{Non-local Cauchy data.}\label{Sec_DN_map}
Let $\Omega \subset \R^n$ be a bounded Lipschitz domain, $n\geq 1$. 
Let us consider $u_f \in H^t(\R^n)$ be a solution of the Dirichlet exterior value problem \eqref{eq2}.
%$\mL_{b,q} u =0$ in $\Omega$ and $u=f$ on $\Omega_e$ for some $b,q \in L^{\infty}(\Omega)$ and $f \in \widetilde{H}^t(\Omega_e)$.
Let us introduce the operator $\mathscr{N}_{b,q}(f)$, we will call it as Neumann data,  using the ``non-local normal derivative" (see \cite{DSRV}) of $u_f$  given
\begin{equation}\label{nd}
\mathscr{N}_{b,q}(f)(x) := C_{n,t}\int_{\Omega} \frac{u_f(x) - u_f(y)}{\lvert x-y\rvert^{n+2t}} \,dy, \quad x\in\Omega_e.
\end{equation}
where $u_f\in H^t(\mathbb{R}^n)$ is a unique weak solution of \eqref{eq2}.

Let us now define the non-local Dirichlet to Neumann map for $0<t<1$. Under the assumption \eqref{eva} we get a unique solution $u_f \in H^t(\R^n)$ of the Dirichlet exterior value problem \eqref{eq2}. We define the Dirichlet to Neumann map ${\Lambda}_{b,q}: \widetilde{H}^t(\Omega_e)\mapsto H^{-t}(\Omega_e)$ by
\begin{equation*}
\langle {\Lambda}_{b,q} f, \psi\rangle := \langle \mL_{b,q}u_f,\psi\rangle_{\Omega} \quad \forall\psi\in \widetilde{H}^t(\Omega_e),
\end{equation*}
where $u_f$ is the unique solution of $\mL_{b,q} u_f=0$ in $\Omega$ and $u_f=f$ on $\Omega_e$.\\
\\
Using the integration by parts formula \eqref{R_int_by_parts} and the fact that $\psi|_{\Omega} = 0$, we get,
\begin{equation}\label{bilinear_1}
\begin{aligned}
\langle {\Lambda}_{b,q} f, \psi\rangle 
= \int_{\R^n}\left((-\D)^{t} u_f\right)\, \psi 
+ \int_{\Omega}b\left((-\D)_{\Omega}^{s/2}u_f\right)\, \left((-\D)_{\Omega}^{s/2}\psi\right)
= \int_{\Omega_e}\left((-\D)^t u_f\right)\psi.
\end{aligned}
\end{equation}
Therefore,
\begin{equation}\label{lm}
{\Lambda}_{b,q} f := (-\D)^t u_f, \quad \mbox{in } \Omega_e. 
\end{equation} 
Next we state the relation between ${\Lambda}_{b,q}(f)$ and the non-local Neumann derivative $\mathscr{N}_{b,q}(f)$ in the following proposition.
\begin{Proposition}\label{ln}\cite{GSU} We have
	\begin{equation*}
{\Lambda}_{b,q} f = \mathscr{N}_{b,q}(f) - mf + (-\Delta)^t(E_0 f)|_{\Omega_e}, \qquad f \in \widetilde{H}^t(\Omega_e)
\end{equation*}
where, for $\gamma > -1/2$, $\mathscr{N}$ is the map 
\begin{equation}\label{Nnew}
\mathscr{N}: H^{\gamma}(\R^n) \to H^{\gamma}_{\mathrm{loc}}(\Omega_e), \ \ \mathscr{N} u = m u|_{\Omega_e} + (-\Delta)^t(\chi_{\Omega} u)|_{\Omega_e}
\end{equation}
where $m \in C^{\infty}(\Omega_e)$ is given by $m(x) = c_{n,t} \int_{\Omega} \frac{1}{|{x-y}|^{n+2t}} \,dy$, $\chi_{\Omega}$ is the characteristic function of $\Omega$ and $E_0$ is extension by zero in $\Omega$.
If $u \in L^2(\R^n)$, then $\mathscr{N} u \in L^2_{\mathrm{loc}}(\Omega_e)$ is given a.e.\ by the formula \eqref{nd}.
\end{Proposition}
\begin{proof}
Since the operators $\Lambda_{b,q}$ (c.f. \eqref{lm}) and $\mathscr{N}$ (c.f. \eqref{nd}) do not involve any regional non-locality from $\Omega$,  the proof is the same as       in \cite[Lemma 3.2]{GSU}.
\end{proof}
\noindent This result shows that knowing $\Lambda_{b,q}f$  is equivalent to knowing $\mathscr{N}_{b,q}(f)$ for $f\in \widetilde{H}^t(\Omega_e)$, since they differ by known quantities which are independent of $b$ and $q$.

Let $W, \widetilde{W} \subset \Omega_e$ be any non-empty open subsets. We define the non-local Cauchy data $\mathcal{C}_{b,q}(W,\widetilde{W})$ (c.f. \eqref{Cauchy}) corresponding to the operator $\mL_{b,q}$ by
\begin{equation}\label{Cauchy_data}
\mathcal{C}_{b,q}(W,\widetilde{W}) = \{ (f,\mathscr{N}_{b,q}(f)|_{\widetilde{W}}) ; f \in \widetilde{H}^t(W) \}.
\end{equation}

Hence, $\mathcal{C}_{b,q}(W,\widetilde{W})$ is determined by the non-local DN map $\Lambda_{b,q}|_{\widetilde{W}}$ applied on $\widetilde{H}^t(W)$.

The main resemblance of the non-local normal derivative with the local normal derivative can be explained using  the following integration by-parts formula \cite{GLX,DSRV} 
\begin{equation}
\int_{\Omega}v(-\Delta)^tw \,dx + \int_{\Omega_e}v\mathscr{N}(w) = \int_{\Omega}w(-\Delta)^tv \,dx + \int_{\Omega_e}w\mathscr{N}(v)
\end{equation}
together with the following limiting equivalence as (see \cite{DSRV}) 
\[  \lim\limits_{t \to 1}\int_{\mathbb{R}^n\setminus\Omega}v\mathscr{N} (w) \, dx = \int_{\partial\Omega} v\frac{\partial w}{\partial \nu}\, \, d\sigma
\]for all $v,w\in \mathscr{S}(\mathbb{R}^n)$; where $\mathscr{N}(w)$ as in \eqref{Nnew} and $\nu$, $d\sigma$ denote the boundary normal vector and the surface measure respectively.

Here we end our discussion on the direct problems for the regional fractional Laplacian operator $(-\D)^a_{\Omega}$ and the non-local operator $\mL_{b,q}$. Next we move into studying the inverse problems of recovering the coefficients $b,q$ from the associated non-local Cauchy data $\mathcal{C}_{b,q}(W,\widetilde{W})$.

\section{Inverse problems}\label{sec3}
Here we address the inverse problem:
\begin{center}
\textit{Does $\mathcal{C}_{b,q}(W,\widetilde{W})$ (c.f. \eqref{Cauchy_data}) uniquely determine $b$ and $q$ in $\Omega$?}
\end{center}
We will answer this question by proving Theorem \ref{Main_Th_1} and Theorem \ref{Main_Th}.
\vspace{5mm}

\noindent To start with, let us consider the operators
\begin{equation*}
\mL_{b_k,q_k}:=\left((-\D)^t + (-\D)^{s/2}_{\Omega}b_k(-\D)^{s/2}_{\Omega} + q_k\right), \quad k=1,2,
\end{equation*}
where $b_k$, $q_k$ are $L^{\infty}(\Omega)$ functions satisfying \eqref{eva}.
%and assume $\mathscr{C}^{W,\widetilde{W}}_{b_1,q_1} = \mathscr{C}_{b_2,q_2}^{W,\widetilde{W}}$.
Let $\mathcal{C}_{b_1,q_1}(W,\widetilde{W})=\mathcal{C}_{b_2,q_2}(W,\widetilde{W})$. Then for any $f \in \widetilde{H}^t(W)$ we have
\begin{equation*}
\mathscr{N}(u_2)|_{\widetilde{W}} = \mathscr{N}(u_1)|_{\widetilde{W}},
\end{equation*}
where
\begin{equation}\label{eq5}
\mL_{b_k,q_k}u_k = 0 \quad\mbox{in }\Omega, \quad \mbox{and}\quad u_k|_{\Omega_e}=f \in \widetilde{H}^t(\Omega_e),\qquad \mbox{for }k=1,2.
\end{equation}
%Hence, we get
%\begin{equation*}
%\mathscr{N}_{b_1,q_1} u_1|_{\widetilde{W}} = \mathscr{N}_{b_2,q_2} u_2|_{\widetilde{W}}.
%\end{equation*}
%Then for any $f \in C^{\infty}_c(W)$ we have $u_k \in H^t(\R^n)$ solutions of
%with $\mathscr{N}_{b_1,q_1}(u_1)|_{\widetilde{W}} = \mathscr{N}_{b_2,q_2}(u_2)|_{\widetilde{W}}$.
%\noindent Let us consider the Dirichlet to Neumann map $\Lambda_{b_k,q_k}$ corresponding to the operator $\mL_{b_k,q_k}$, for $k=1,2$.
%Consider the 
%Then $u_1=u_2$ in $\R^n$.

\begin{Lemma}[Integral Identity]\label{Int_id}
Let $f \in \widetilde{H}^t(\Omega_{e})$ and $u_k \in H^t(\R^n)$ be solutions of \eqref{eq5} for $k=1,2$. 
Let $\widetilde{W} \subset \Omega_e$ be a non-empty open set such that
\begin{equation*}
\mathscr{N}(u_1)|_{\widetilde{W}} = \mathscr{N}(u_2)|_{\widetilde{W}}.
\end{equation*}
\noindent Then $u_1 = u_2$ in $\R^n$.

\noindent Moreover, we have the following integral identity
\begin{equation}\label{int_id_2}
\int_{\Omega}(b_1-b_2)\left((-\D)^{s/2}_{\Omega} u_f\right) \left((-\D)^{s/2}_{\Omega}\varphi\right) + \int_{\Omega} (q_1 - q_2)\varphi\, u_f =0, \quad\forall\varphi \in C^{\infty}_c(\Omega),
\end{equation}
where $u_f = u_1 = u_2$ in $\R^n$.
\end{Lemma}
%The proof of Lemma \ref{Int_id} 

%From here on we move into the study of the inverse problem for the operator $\mL_{b,q}$. 
%In the direct problem (Section \ref{Sec_2}) we see that knowing the coefficients $b$ and $q$ one can determine the Dirichlet to Neumann map $\Lambda_{b,q}$ in the exterior domain.
%Here we analyze the inverse problem which is to determine the coefficients $b$ and $q$ from the knowledge of the Dirichlet to Neumann map $\Lambda_{b,q}$ in the exterior domain.
%In this article we deal with two types of inverse problems discussed in the introduction (Section \ref{IP}). 
%In order to solve the inverse problems in Theorem \ref{Main_Th_1} and \ref{Main_Th} we first develop the following unique continuation and Runge type approximation properties for the non-local operators we are dealing with.
\noindent Next we prove Lemma \ref{Int_id} which follows from the unique continuation property for the fractional Laplacian operator (see \cite[Theorem 1.2]{GSU})
\begin{Proposition}\label{sucp}
Let $u\in H^{-r}(\mathbb{R}^n)$, $r>0$.
If $u= (-\Delta)^tu = 0$ in some non-empty open set $\mathcal{O}\subset\mathbb{R}^n$, then $u\equiv 0$ in $\R^n$.  
\end{Proposition}

\begin{proof}[Proof of Lemma \ref{Int_id}:]
%Let us assume $\widetilde{W}, \mL_{b_k,q_k}, \Lambda_{b_k,q_k}$ and $u_k$ for $k=1,2$, are as in the statement of the proposition.
By using the Proposition \ref{ln} and \eqref{lm} from $\mathscr{N}_{b,q}(f)|_{\widetilde{W}} = \mathscr{N}_{b,q}(f)|_{\widetilde{W}}$, we obtain $(-\Delta)^t u_1|_{\widetilde{W}} = (-\Delta)^t u_2|_{\widetilde{W}}$.
Since $u_1=u_2=f$ in $\Omega_e$,
so we have 
\begin{equation*}
(-\Delta)^t(u_1-u_2)|_{\widetilde{W}} = 0 = (u_1-u_2)|_{\widetilde{W}}.
\end{equation*}
Therefore, from Proposition \ref{sucp} it follows that
$u_1=u_2 \mbox{ on }\R^n$.

Let us now denote $u_f=u_1=u_2$ in $\R^n$ and observe that
\begin{equation*}
\begin{aligned}
\left((-\D)^t + (-\D)^{s/2}_{\Omega}b_1(-\D)^{s/2}_{\Omega} + q_1\right)u_f &=0, \quad\mbox{in }\Omega,\\
\left((-\D)^t + (-\D)^{s/2}_{\Omega}b_2(-\D)^{s/2}_{\Omega} + q_2\right)u_f &=0. \quad\mbox{in }\Omega,
\end{aligned}
\end{equation*}
which implies
\begin{equation}\label{int_id_1}
(-\D)^{s/2}_{\Omega}(b_1-b_2)(-\D)^{s/2}_{\Omega} u_f + (q_1 - q_2)u_f =0, \quad\mbox{in }\Omega.
\end{equation}
One can equivalently write the above equation as
 \begin{equation*}
\int_{\Omega}(b_1-b_2)\left((-\D)^{s/2}_{\Omega} u_f\right) \left((-\D)^{s/2}_{\Omega}\varphi\right) + \int_{\Omega} (q_1 - q_2)\varphi\, u_f =0, \quad\forall\varphi \in C^{\infty}_c(\Omega).
\end{equation*}
\end{proof}

Now our goal is to show $b_1=b_2$ and $q_1=q_2$ from the integral identity \eqref{int_id_2}.
In order to do that we derive a unique continuation property for the regional fractional Laplacian and Runge approximation property for the non-local operators $(-\D)^a_{\Omega}$ and $\mL_{b,q}$.
\subsection{Unique continuation and Runge approximation}
As a consequence of Proposition \ref{sucp} we get the following Runge approximation property (see \cite[Theorem 1.3]{GSU}).
Let us consider the set
\[ X_{\mathcal{O},W}:=\{v|_{\mathcal{O}}\,;\,\, (-\Delta)^tv=0, \mbox{ in }\mathcal{O},\,\, v|_{\mathcal{O}_e}=f,\, \forall f\in C^\infty_c(W)\}\]
where $W$ be some open bounded subset of $\mathcal{O}_e:= \R^n \setminus \overline{\mathcal{O}}$.
\begin{Proposition}\label{oldrunge}
The set $X_{\mathcal{O},W}$ is dense in $L^2(\mathcal{O})$. 
\end{Proposition}

%We exploit the strong unique continuation property of the fractional Laplacian $(-\Delta)^t$, $0<t<1$ and apply it to achieve various approximation results.
\noindent For the regional fractional Laplacian $(-\Delta)^a_{\Omega}$, $0<a<1$, first we prove the following unique continuation property.

\begin{Lemma}\label{sucpl}
Let $\Omega \subset \R^n$, be a bounded Lipschitz domain. Let $v\in H^{a}(\Omega)$, $0<a<1 $. If $v=(-\Delta)^a_{\Omega}v = 0$ on a non-empty open subset $\mathcal{O} \Subset \Omega$, then $v=0$ in $\Omega$.
\end{Lemma}
\begin{proof}
\textbf{Case 1:} Let us take $v\in \widetilde{H}^a(\Omega)$ and extend it by zero in $\Omega_e$. From \eqref{local_nonlocal} with using the fact $v=(-\Delta)^a_{\Omega}v = 0$ in $\mathcal{O}$ we simply obtain 
\begin{equation*}
v=(-\Delta)^av = 0 \mbox{ in } \mathcal{O}.
\end{equation*}
Consequently, from Proposition \ref{sucp} we obtain $v\equiv 0$, or $v= 0$ in $\Omega$.

\textbf{Case 2:} Let $v\in H^a(\Omega)$, and $0 < a < \frac{1}{2}$. Then using the fact that $\widetilde{H}^a(\Omega)=H^a(\Omega)$ for $0<a < \frac{1}{2}$ (see \eqref{key_1}, \eqref{key2}) we get $v \in \widetilde{H}^a(\Omega)$ and using \textbf{Case 1} we prove the lemma.

\textbf{Case 3:} Now if $v\in H^a(\Omega)$, and $a\geq\frac{1}{2}$, then we define $\widetilde{v}=\chi_{\Omega}v$ in $\mathbb{R}^n$ to have at least $\widetilde{v}\in L^2(\mathbb{R}^n)$. Note that, $(-\Delta)^a_{\Omega}v=(-\Delta)^a_{\Omega}\widetilde{v}=0$ in $\Omega$. 
Now, from \eqref{local_nonlocal} we obtain $(-\Delta)^a\widetilde{v} = 0 \mbox{ in } \mathcal{O}$.
Since the Proposition \ref{sucp} is also valid for $L^2(\mathbb{R}^n)$ functions, so from $\widetilde{v}=(-\Delta)^a\widetilde{v}=0$ in $\mathcal{O}\Subset\Omega$, implies $\widetilde{v}\equiv 0$, or $v=0$ in $\Omega$.  
\end{proof}
Using the unique continuation Lemma \ref{sucpl}, we prove the following Runge approximation result for the regional fractional Laplacian.  
%\subsection{Runge approximation for $(-\D)^a_{\Omega}$} 
\begin{Lemma}\label{runge_0}
Let $\Omega$ be a bounded Lipschitz domain in $\mathbb{R}^n$, and  $\mathcal{O}\Subset \Omega$ be a non-empty open subset,  with Lipschitz boundary, compactly contained in $\Omega$. Then for $0<a<1$,  we have
\begin{equation*}
X_{\mathcal{O}}:= \{ v|_{\mathcal{O}} ; v\in {H}^a(\Omega),\,  (-\Delta)^a_{\Omega} v = 0 \mbox{ in }\mathcal{O} \}
\end{equation*}
is dense in $L^2(\mathcal{O})$.
\end{Lemma}

\begin{proof}%[Proof of Lemma \ref{runge_0}]
It is enough to show that
\begin{equation}\label{condi11}
\mbox{if } \langle w, v \rangle_{L^2(\mathcal{O})} = 0 \quad \forall v\in X_{\mathcal{O}},\quad \mbox{then } w=0 \mbox{ in }\mathcal{O}.
\end{equation}
Let us assume that, there is a $ w \in L^2(\mathcal{O})$ such that $\langle w,v\rangle_{L^2(\mathcal{O})} = 0$ for all $v \in X_{\mathcal{O}}$.
Then we consider the function $\varphi\in \widetilde{H}^{a}(\mathcal{O})$ as the unique weak solution of 
\begin{equation}\label{phiw1}\begin{aligned} (-\D)^a_{\Omega} \varphi &= w \quad\mbox{in }\mathcal{O}\\
\varphi &= 0 \quad\mbox{in }\Omega\setminus\overline{\mathcal{O}}.
           \end{aligned}\end{equation}
Now from \eqref{condi11} and \eqref{phiw1} we have
\begin{align*}
0 = \langle w,v \rangle_{L^2(\mathcal{O})} 
&= \langle (-\Delta)^a_{\Omega} \varphi, v\rangle_{L^2(\mathcal{O})}\\
&=  \langle (-\Delta)^a_{\Omega} \varphi, v\rangle_{L^2(\Omega)}-\langle (-\Delta)^a_{\Omega} \varphi, v\rangle_{L^2(\Omega\setminus\overline{\mathcal{O}})}.
\end{align*}
Thus for all $v\in X_{\mathcal{O}}$, we conclude
\begin{equation}\label{qwe1}
\begin{aligned}
 \langle (-\Delta)^a_{\Omega} \varphi, v\rangle_{L^2(\Omega\setminus\overline{\mathcal{O}})} = \langle (-\Delta)^a_{\Omega} \varphi, v\rangle_{L^2(\Omega)} = \langle\varphi, (-\Delta)^a_{\Omega}v\rangle_{L^2(\mathcal{O})}
 =0.
\end{aligned}
\end{equation} 
Since $v|_{\Omega\setminus\overline{\mathcal{O}}}\in H^a(\Omega\setminus\overline{\mathcal{O}})$ can be chosen arbitrarily, \eqref{qwe1} implies
\begin{equation}\label{nw1}
(-\Delta)^a_{\Omega}\varphi = 0 \quad\mbox{ in }\Omega\setminus\overline{\mathcal{O}}. 
\end{equation}
Therefore, we obtain that 
\[ \varphi= (-\D)^a_{\Omega} \varphi = 0 \mbox{ in }\Omega\setminus\overline{\mathcal{O}},\]
which implies that $\varphi\equiv 0 $ in $\Omega$ due to the unique continuation (cf. Lemma \ref{sucpl}) and hence $w=0$ in $\mathcal{O}$. This proves \eqref{condi11} and Lemma \ref{runge_0}.
\hfill\end{proof}
Now we prove a Runge approximation type property for the operator $\mL_{b,q}$.
%The following result provides direct recovery of the coefficients $b$ and $q$ from all measurements of the Cauchy data.
Let us recall the operator $\mL_{b,q}$ introduced in \eqref{operator}. Let $\mathcal{O} \Subset \Omega$ be a non-empty open set, compactly contained in $\Omega$, with Lipschitz boundary and consider the sets
\begin{equation*}
X:=\{(-\Delta)^{s/2}_{\Omega}v|_{\mathcal{O}}\,;\,\, \mathscr{L}_{b,q}v=0, \mbox{ in }\Omega,\,\, v|_{\Omega_e}=f,\, \forall f\in C^\infty_c(W)\}
\end{equation*}
and
\begin{equation*}
Y:=\{v|_{\Omega}\,;\,\, \mathscr{L}_{b,q}v=0, \mbox{ in }\Omega,\, v|_{\Omega_e}=f,\,\, \forall f\in C^\infty_c(W)\}
\end{equation*}
where $W$ be some non-empty open bounded subset of $\Omega_e$.
\begin{Lemma}\label{sqml}
Let $\mL_{b,q}, X,Y$ are as above. Then
\begin{enumerate}
 \item For any $F\in L^2(\mathcal{O})$ and any $\epsilon>0$ there exists $u\in X$ such that
\begin{equation*}
\|F-u\|_{L^2(\mathcal{O})}<\epsilon.
\end{equation*}
\item 
The set $Y$ is dense in $L^2(\Omega)$. 
\end{enumerate}
\end{Lemma}
\begin{proof}
(1)\,
We observe that, it is enough to prove the result for $F \in \widetilde{H}^{s}(\mathcal{O})$ for some $0<s<1$, since $\widetilde{H}^{s}(\mathcal{O})$ is dense in $L^2(\mathcal{O})$.
Let $F \in \widetilde{H}^{s}(\mathcal{O})$ such that $\langle F, \widetilde{v}\rangle_{L^2(\mathcal{O})}= 0$, for all $\widetilde{v}\in X$, then we show $F=0$ in $\mathcal{O}$. Since $\widetilde{v}=(-\Delta)^{s/2}_\Omega v|_{\mathcal{O}} \in L^2(\mathcal{O})$ for some  $v\in H^t(\mathbb{R}^n)$ solving $\mL_{b,q}v= 0 $ in $\Omega$.
Thus we get that \[\langle F , (-\Delta)^{s/2}_\Omega v\rangle_{L^2(\mathcal{O})} = 0.\]

Now extending $F$ by $0$ outside $\mathcal{O}$ we have $F \in \widetilde{H}^{s}(\Omega)$ and consequently $(-\Delta)^{s/2}_{\Omega}F \in L^2(\Omega)$ thanks to Lemma \ref{lemnew}. %It follows from the fact 
%$(-\Delta)^{s/2}_{\Omega}F +\varphi_{s/2}(x)F = (-\Delta)^{s/2}F\in L^2(\Omega)$. Furthermore recalling that $\varphi_a$ is locally Lipschitz and that $F=0$ in $\Omega\setminus \mathcal{O}$ we have $\varphi_{s/2}(x)F|_{\Omega} \in L^2(\Omega)$, which establishes our claim. 
Next we write
\[0=\langle F , (-\Delta)^{s/2}_\Omega\, v\rangle_{L^2(\mathcal{O})} = \left\langle \left( (-\Delta)^{s/2}_\Omega\, F\right) , v\right\rangle_{L^2(\Omega)}.\]
Since $(-\Delta)^{s/2}_{\Omega}F \in L^2(\Omega)$, there is $w\in H^t(\R^n)$, $0<t<1$ such that
\begin{equation*}
\mL_{b,q}w = (-\Delta)^{s/2}_{\Omega}F\quad \mbox{in } \Omega, \quad \mbox{with}\quad w=0 \mbox{ in }\Omega_e.
\end{equation*}
Therefore, we get
\begin{equation*}
0= \langle \mL_{b,q}w, v\rangle_{L^2(\Omega)} = \langle w, \mL_{b,q}v\rangle_{L^2(\Omega)} -\langle (-\Delta)^tw, v\rangle_{L^2(\Omega_e)}.
\end{equation*}
Since $\mL_{b,q}v=0$ in $\Omega$, thus
\begin{equation*}
\langle (-\Delta)^tw, f\rangle_{\Omega_e}
%=\langle (-\Delta)^tw, v\rangle_{\Omega_e} 
= 0, \quad\forall f\in C^\infty_c(W).
\end{equation*}
Hence, $(-\Delta)^tw=0=w $ in $W\subset\Omega_e$. Consequently, by the unique continuation given by Proposition \ref{sucp} we conclude that $w\equiv 0$, that $(-\Delta)^{s/2}_\Omega F = 0$ in $\Omega$. Using the unique continuation for the regional fractional Laplacian operator since $(-\D)^{s/2}_{\Omega}F = 0 = F$ in $\Omega\setminus \overline{\mathcal{O}}$, we get $F=0$ in $\Omega$ (c.f. Lemma \ref{sucpl}). Hence, part (1) follows. \\
\\
\noindent (2)\, The proof is similar to part (1). Let $G\in L^2(\Omega)$ and 
$ \langle G, v\rangle_{L^2(\Omega)}= 0$, for all $v\in Y$, then we will show that $G=0$ in $\Omega$ to prove our claim.  

Let $w\in H^t(\mathbb{R}^n)$ solves
$
\mL_{b,q}w = G$ in $\Omega$ and $w=0$ in $\Omega_e.
$
Then we have
\begin{equation*}
0= \langle \mL_{b,q}w, v\rangle_{L^2(\Omega)} = \langle w, \mL_{b,q}v\rangle_{L^2(\Omega)} -\langle (-\Delta)^tw, v\rangle_{L^2(\Omega_e)}.
\end{equation*}
Since $\mL_{b,q}v=0$ in $\Omega$, we get
\begin{equation*}
\langle (-\Delta)^tw, f\rangle_{\Omega_e}=\langle (-\Delta)^tw, v\rangle_{\Omega_e} = 0, \quad \mbox{for all } f\in C^\infty_c(W).
\end{equation*}
Hence, $(-\Delta)^tw=0=w $ in $W\subset\Omega_e$. Consequently, by the unique continuation in Proposition \ref{sucp} we have $w\equiv 0$ and therefore $G=0$ in $\Omega$. 
\hfill\end{proof}

\subsection{Proof of Theorem \ref{Main_Th_1} and \ref{Main_Th}:}\label{Sec_inv_prob}
Let us recall the integral identity from Lemma \ref{Int_id}, given by
\begin{equation}\label{newline1}
\int_{\Omega}(b_1-b_2)\left((-\D)^{s/2}_{\Omega} u_f\right) \left((-\D)^{s/2}_{\Omega}\varphi\right) + \int_{\Omega} (q_1 - q_2)\varphi\, u_f =0, \quad\forall\varphi \in C^{\infty}_c(\Omega),
\end{equation}
where $u_f \in H^t(\R^n)$ is the unique solution of the Dirichlet exterior value problem $\mL_{b_k,q_k}u_f = 0$ in $\Omega$ and $u_f = f$ in $\Omega_e$, for $k=1,2$.

Since $b_k$, $q_k$ are compactly supported in $\Omega$, let us consider a non-empty open subset $\mathcal{O}\Subset\Omega$,  with Lipschitz boundary, containing the compact supports of $b_k,q_k$ in $\Omega$, for $k=1,2$.
Now, using Lemma \ref{runge_0} we know that
\begin{equation*}
X_{\mathcal{O}}=\{\varphi|_{\mathcal{O}}:\, \varphi\in H^{s/2}(\Omega);\, (-\Delta)^{s/2}_{\Omega}\varphi=0 \mbox{ in }\mathcal{O}\}
\end{equation*}
is dense in $L^2(\mathcal{O})$. Hence, from the $L^2$-density of  $C^\infty_c(\Omega)$ in $X_{\mathcal{O}}$ and the integral identity \eqref{newline1}, we conclude   
\begin{equation*}
\int_{\mathcal{O}} (q_1-q_2)u_f\,\varphi = 0,\quad \forall \varphi\in X_{\mathcal{O}},
\end{equation*}
or equivalently
\begin{equation}\label{imp1}
(q_1 - q_2)u_f = 0 \quad \mbox{in }\mathcal{O}.
\end{equation}
Plugging \eqref{imp1} in the integral identity \eqref{newline1} we have
\begin{equation}\label{newid2}
\int_{\Omega} (b_1-b_2) \left((-\Delta)^{s/2}_{\Omega} u_f\right)\left((-\Delta)^{s/2}_{\Omega} \varphi\right) = 0, \quad \forall\varphi \in C^{\infty}_c(\Omega).
\end{equation}
Note that $\mathcal{O}\Subset\Omega$ contains the compact supports of $b_1,b_2$ in $\Omega$ and $(b_1-b_2) (-\Delta)^{s/2}_{\Omega} u_f \in L^2(\mathcal{O})$.
Hence, take 
$\varphi \in H^{s/2}(\Omega)$ to be the weak solution of (see Proposition \ref{existence_th2}) 
\begin{equation*}
\begin{aligned}
(-\Delta)^{s/2}_\Omega \varphi =& (b_1-b_2) (-\Delta)^{s/2}_{\Omega} u_f\quad \mbox{ in }\mathcal{O}\\
\varphi =& 0 \quad \mbox{on }\Omega\setminus\overline{\mathcal{O}}.
\end{aligned}
\end{equation*}
Then plugging this $\varphi$ in \eqref{newid2}, we obtain 
\begin{equation}\label{imp2}
(b_1-b_2) (-\Delta)^{s/2}_{\Omega} u_f = 0 \quad \mbox{in }\Omega.
\end{equation}
Observe the relations obtained in \eqref{imp1} and \eqref{imp2}, that is
\begin{equation}\label{identity}
(b_1-b_2) (-\Delta)^{s/2}_{\Omega} u_f = 0 = (q_1-q_2)u_f \quad \mbox{in }\Omega.
\end{equation}
Note that the above identity \eqref{identity} is true for each $f \in \widetilde{H}^t(\Omega_e)$ (and the corresponding solution $u_f \in H^t(\R^n)$) satisfying $\mathscr{N}_{b,q}(f)|_{\widetilde{W}} = \mathscr{N}_{b,q}(f)|_{\widetilde{W}}$.

\subsection*{Proof of Theorem \ref{Main_Th}}
Let us assume that the operators $\mL_{b_k,q_k}$, for $k=1,2$ satisfies the assumption \eqref{eva}.
Fix a nonzero $f \in \widetilde{H}^t(\Omega_e)$ such that \[\left(\Lambda_{b_1,q_1} - \Lambda_{b_2,q_2}\right)(f)|_{\widetilde{W}} = 0.\]
%Observe that the identity \eqref{identity} holds true for this particular $f$ and corresponding $u_f \in H^t(\R^n)$.
Since $(b_1-b_2), (q_1-q_2) \in C_c(\Omega)$, then \begin{equation}\label{bc}B=\{x\in \Omega;\, (b_1-b_2)(x)\neq 0\}\quad\mbox{ and }\quad C=\{x\in \Omega;\, (q_1-q_2)(x)\neq 0\}\end{equation} are open subsets in $\Omega$. If $B,C$ are non-empty, then from \eqref{identity}, we get $(-\Delta)^{s/2}_{\Omega}u_f$ and $u_f$  are zero on the open sets $B$ and $C$ respectively.  

From Lemma \ref{sucpl}, it is clear that $B$ and $C$ are disjoint open sets. If $B\cap C\neq \emptyset$, it will lead to $u_f\equiv 0$ in $\Omega$.
We also observe that $B$, $C$ cannot be complementary in $\Omega$, that if $C=\Omega\setminus \overline{B}$ %\overline{B}\cup C\neq\Omega$, since 
then the following exterior value problem 
\[ (-\Delta)^{s/2}_{\Omega}v= 0 \mbox{ in }B, \quad v = 0 \mbox{ in }\Omega\setminus\overline{B}\]
has only $v=0$ solution in $\Omega$ (cf. Proposition \ref{existence_th2}).
This completes the proof of Theorem \ref{Main_Th}.
\qed

\subsection*{Proof of Theorem \ref{Main_Th_1}:}
Let $\mathscr{N}_{b,q}(f)|_{\widetilde{W}} = \mathscr{N}_{b,q}(f)|_{\widetilde{W}}$, for all $f \in \widetilde{H}^t(W)$, where $W, \widetilde{W}\subset \Omega_e$ be any non-empty open sets.
Then we have the identity \eqref{identity}:
\begin{equation*}
(b_1-b_2) (-\Delta)^{s/2}_{\Omega} u_f=0 = (q_1-q_2)u_f \quad \mbox{in }\Omega,\quad \forall f\in C^\infty_c(W).
\end{equation*}
Now let us recall the approximation result of Lemma \ref{sqml}.
Since, $(b_1-b_2)$ is compactly supported inside $\Omega$, using the density result concerning $(-\Delta)^{s/2}_\Omega u_f$ in the \textit{Part (1)} of Lemma \ref{sqml} and varying  $f\in C^\infty_c(W)$ we obtain $b_1=b_2$ in $\Omega$.
Similarly by using \textit{Part (2)} of Lemma \ref{sqml} we get $q_1=q_2$ in $\Omega$.
This completes the proof of Theorem \ref{Main_Th_1}.
\qed

\section*{Acknowledgment}
\begin{small}
	\noindent G.U. was partly supported by NSF, the Si-Yuan Professorship at the IAS-HKUST and the Walker Family Endowed Professorship at the University of Washington. S.B. and T.G. were partly supported by Project no.: 16305018 of the Hong Kong Research Grant Council.
\end{small}

%\bibliographystyle{alpha}
%\bibliography{bib_frac}
%\end{document}

\end{document}